\documentclass[a4paper]{amsart}
\usepackage[T1]{fontenc}
\usepackage[utf8]{inputenc}
\usepackage[british]{babel}
\usepackage{amsfonts,amsmath,amssymb,amsthm,mathtools,esint}
\usepackage[dvipsnames]{xcolor}
\usepackage{tikz}
\usepackage{pgfplots}
\pgfplotsset{compat=1.3}
\usepackage[justification=centering]{caption}
\usepackage{subcaption}
\usepackage{enumitem}
\usepackage[hidelinks]{hyperref}
\usepackage{cleveref}
\crefname{subsection}{subsection}{subsections}
\numberwithin{equation}{section}

\newcommand{\D}{\mathrm{D}}
\renewcommand{\d}{\,\mathrm{d}}

\newtheorem{theorem}{Theorem}[section]
\newtheorem{proposition}[theorem]{Proposition}
\newtheorem{lemma}[theorem]{Lemma}
\newtheorem{definition}[theorem]{Definition}
\newtheorem{corollary}[theorem]{Corollary}
\theoremstyle{remark}
\newtheorem{remark}[theorem]{Remark}

\begin{document}

\title[Guaranteed error control for Monge--Amp\`ere]
      {Stability and guaranteed error control of approximations
      to the Monge--Amp\`ere equation}

\author[D.~Gallistl \and N.~T. Tran]
{Dietmar Gallistl \and Ngoc Tien Tran}
\thanks{This project received funding from
	the European Union's Horizon 2020 research and innovation 
	programme (project DAFNE, grant agreement No.~891734).
	}
\address[D.~Gallistl, N.~T.~Tran]{%
	Institut f\"ur Mathematik,
	Friedrich-Schiller-Uni\-ver\-si\-t\"at Jena,
	07743 Jena, Germany}
\email{ $\{$dietmar.gallistl, ngoc.tien.tran$\}$@uni-jena.de}
\date{\today}



\begin{abstract}
	This paper analyzes a regularization scheme of the Monge--Amp\`ere equation
	by uniformly elliptic Hamilton--Jacobi--Bellman equations.
	The main tools are stability estimates in the $L^\infty$ norm 
    from the theory of viscosity solutions which are independent of
	the regularization parameter $\varepsilon$.
	They allow for the uniform convergence of the solution $u_\varepsilon$ 
	to the regularized problem towards the Alexandrov solution $u$ to the 
	Monge--Amp\`ere equation for any nonnegative $L^n$ right-hand side
	and continuous Dirichlet data.
	The main application are guaranteed a~posteriori error bounds in the $L^\infty$ norm
	for continuously differentiable finite element approximations 
	of $u$ or $u_\varepsilon$.
\end{abstract}

\maketitle

\section{Introduction}

\subsection*{Overview}
Let $\Omega\subset\mathbb R^n$, $n\geq 2$, be a bounded and convex domain.
Given a nonnegative function $0 \leq f \in L^n(\Omega)$ and continuous 
Dirichlet data $g \in C(\partial\Omega)$, the Monge--Amp\`ere equation
seeks the unique (convex) Alexandrov solution $u \in C(\overline{\Omega})$ to
\begin{align}\label{def:Monge-Ampere}
	\det \D^2 u = (f/n)^n \text{ in } \Omega \quad\text{and}\quad u = g \text{ on } \partial \Omega.
\end{align}
If the Dirichlet data $g \neq 0$ is non-homogenous, 
then we additionally assume that $\Omega$ is strictly convex.
The re-scaling $\tilde f \coloneqq (f/n)^n$ of the right-hand side is not essential,
but turns out convenient for purposes of notation.
By the Alexandrov solution $u$ to \eqref{def:Monge-Ampere} we mean 
a convex function $u \in C(\Omega)$ with $u=g$ on $\partial\Omega$
and
$$
   \mathcal{L}^n(\partial v(\omega)) = \int_\omega \tilde f\d x
   \quad\text{for any Borel subset }\omega \subset\Omega.
$$
The left-hand side denotes the Monge--Amp\`ere measure of $\omega$,
i.e., the $n$-dimensional Lebesgue measure of
all vectors in the subdifferential 
$\partial v(\omega) \coloneqq \cup_{x \in \omega} \partial v(x)$
where $\partial v(x)$ is the usual subdifferential of $v$ in a point $x$.
We remark that this solution concept admits more general right-hand
sides, which are, however, not disregarded in this work.
For further details, we refer to the monographs \cite{Gutierrez2016,Figalli2017}.
It is known \cite{Alexandrov1958} that the Alexandrov solution to
\eqref{def:Monge-Ampere} exists and is unique.
In addition, it was shown \cite{Caffarelli1990} that
if $f \in C^{0,\alpha}(\Omega)$, $0 < \lambda \leq f \leq \Lambda$, 
and $g \in C^{1,\beta}(\partial\Omega)$ with positive constants $0 < \alpha,\beta < 1$ and $0 < \lambda \leq \Lambda$, 
then $u \in C(\overline{\Omega}) \cap C^{2,\alpha}_{\text{loc}}(\Omega)$.

It is known \cite{Krylov1987,FengJensen2017} that \eqref{def:Monge-Ampere} 
can be equivalently formulated as a Hamilton--Jacobi--Bellman (HJB)
equation, a property that turned out useful for the numerical solution
of \eqref{def:Monge-Ampere} \cite{FengJensen2017,gt2021};
one of the reasons being that the 
latter is elliptic on the whole space of symmetric matrices 
$\mathbb{S} \subset \mathbb{R}^{n \times n}$ and, therefore,
the convexity constraint is automatically
enforced by the HJB formulation.
For nonnegative continuous right-hand sides $0 \leq f \in C(\Omega)$, the Monge--Amp\`ere equation \eqref{def:Monge-Ampere} is equivalent to
\begin{align*}
	F_0(f;x,\D^2 u) = 0 \text{ in } \Omega \quad\text{and}\quad u = g \text{ on } \partial \Omega
\end{align*}
with $F_0(f;x,M) \coloneqq \sup_{A \in \mathbb{S}(0)}(-A:M + f\sqrt[n]{\det A})$ for any $x \in \Omega$ and $M \in \mathbb{R}^{n\times n}$. Here, $\mathbb{S}(0) \coloneqq \{A \in \mathbb{S}: A \geq 0 \text{ and } \mathrm{tr}\,A = 1\}$ denotes the set of positive semidefinite symmetric matrices $A$ with unit trace $\mathrm{tr}\,A = 1$.
Since $F_0$ is only degenerate elliptic, the regularization scheme proposed in \cite{gt2021} replaces $\mathbb{S}(0)$ by a compact subset $\mathbb{S}(\varepsilon) \coloneqq \{A \in \mathbb{S}(0): A \geq \varepsilon\} \subset \mathbb{S}(0)$ of matrices with eigenvalues bounded from below by the regularization parameter $0 < \varepsilon \leq 1/n$.
The solution $u_\varepsilon$ to the regularized PDE solves
\begin{align}\label{def:HJB}
	F_\varepsilon(f;x,\D^2 u_\varepsilon) = 0 \text{ in } \Omega \quad\text{and}\quad u_\varepsilon = g \text{ on } \partial \Omega
\end{align}
where, for any $x \in \Omega$ and $M \in \mathbb{R}^{n \times n}$, 
the function $F_\varepsilon$ is defined as
\begin{align}\label{def:F-eps}
	F_\varepsilon(f;x,M) \coloneqq \sup\nolimits_{A \in \mathbb{S}(\varepsilon)} (-A:M + f\sqrt[n]{\det A}).
\end{align}
In two space dimensions $n = 2$, uniformly elliptic HJB equations satisfy the Cordes condition 
\cite{MaugeriPalagachevSoftova2000}
and this allows for a variational setting for \eqref{def:HJB} with a unique strong solution $u_\varepsilon \in H^2(\Omega)$ in the sense that $F_\varepsilon(f;x,\D^2 u_\varepsilon) = 0$ holds a.e.~in $\Omega$ \cite{SmearsSueli2013,SmearsSueli2014}.
The paper \cite{gt2021} establishes uniform convergence of $u_\varepsilon$ towards the generalized solution $u$ to the Monge--Amp\`ere equation \eqref{def:Monge-Ampere} as $\varepsilon \searrow 0$ under the assumption $g \in H^2(\Omega) \cap C^{1,\alpha}(\overline{\Omega})$ and that $0 \leq f \in L^2(\Omega)$ can be approximated from below by a pointwise monotone sequence of positive continuous functions.

\subsection*{Contributions of this paper}
The variational setting of \eqref{def:HJB} in two space dimensions leads to $H^2$ stability estimates that deteriorate with $\varepsilon^{-1} \to \infty$ as the regularization parameter $\varepsilon \to 0$ vanishes.
This can be explained by the regularity of Alexandrov solutions to the Monge--Amp\`ere equation \eqref{def:Monge-Ampere} as they are, in general, not in $H^2(\Omega)$ without additional assumptions on the domain $\Omega$ and the data $f,g$.
Consequently, error estimates in the $H^2$ norm may not be of interest, and the focus is on error estimates in the $L^\infty$ norm.

The analysis departs from the following $L^\infty$ stability estimate that arises from the Alexandrov maximum principle.
If $v_1,v_2 \in C(\overline{\Omega})$ are viscosity solutions to $F_\varepsilon(f_j; x, \D^2 v_j) = 0$ in $\Omega$ with $0 \leq \varepsilon \leq 1/n$ and $f_1,f_2 \in C(\overline{\Omega})$, then
\begin{align}\label{ineq:stability-intro}
	\|v_1-v_2\|_{L^\infty(\Omega)} \leq \|v_1-v_2\|_{L^\infty(\partial \Omega)} + C(n,\mathrm{diam}(\Omega))\|f_1-f_2\|_{L^n(\Omega)}.
\end{align}
The constant $C(n,\mathrm{diam}(\Omega))$ exclusively depends
on the dimension $n$ and the diameter $\mathrm{diam}(\Omega)$ of $\Omega$, but 
not on the ellipticity constant of \eqref{def:HJB} or on the regularization parameter $\varepsilon$.
Consequently, this allows for 
control of the $L^\infty$ error even as $\varepsilon \to 0$.
By density of $C(\overline{\Omega})$ in $L^n(\Omega)$, 
the $L^\infty$ stability estimate \eqref{ineq:stability-intro} can be extended to
solutions $v_1,v_2 \in C(\overline{\Omega})$ for $0 < \varepsilon \leq 1/n$ 
(or $\varepsilon = 0$ if $f_1,f_2 \geq 0$) with
the following two applications.
First, this paper establishes, in extension to \cite{gt2021}, uniform convergence of 
(generalized) 
viscosity solutions $u_\varepsilon$ of the regularized PDE \eqref{def:HJB} to the Alexandrov solution 
$u \in C(\overline{\Omega})$ of the Monge--Amp\`ere equation \eqref{def:HJB} 
under the (essentially) minimal assumptions $0 \leq f \in L^n(\Omega)$ and $g \in C(\partial\Omega)$ 
on the data.
Second, \eqref{ineq:stability-intro} provides guaranteed error control in the $L^\infty$ norm (even for inexact solve) for $H^2$ conforming FEM.

\subsection*{Outline}
The principal tool we use for establishing our results 
is
the celebrated Alexandrov maximum principle.
It provides an upper bound for the $L^\infty$ norm of any convex function 
in dependence of its Monge--Amp\`ere measure.
\begin{lemma}[Alexandrov maximum principle]\label{lem:Alexandrov-maximum}
	There exists a constant $c_n$ solely depending on the dimension $n$ such that 
	any convex function $v \in C(\overline{\Omega})$ with homogenous boundary data $v|_{\partial \Omega} = 0$ over an open bounded convex
	domain $\Omega$ satisfies
	\begin{align}
		|v(x)|^n \leq c_n^n\mathrm{dist}(x,\partial \Omega) \mathrm{diam}(\Omega)^{n-1}\mathcal L^n(\partial v(\Omega)) \quad\text{for any } x \in \Omega.
	\end{align}
\end{lemma}
\begin{proof}
	This is \cite[Theorem 2.8]{Figalli2017} and the constant $c_n \coloneqq (2(2\pi)^{n/2-1}/((n-1)!!n)$ arises therein from the $n$-dimensional volume formula for a cone $\mathcal{C} \subset \partial v(\Omega)$. If $n = 2$, then $c_2 = 1$.
\end{proof}

The remaining parts of this paper are organized as follows.
\Cref{sec:HJB} establishes $L^\infty$ stability estimates for viscosity solutions to the HJB equation \eqref{def:HJB} for all parameters $0 \leq \varepsilon \leq 1/n$ in any space dimension.
\Cref{sec:convergence} provides a proof of convergence of the regularization scheme.
A~posteriori error estimates for the discretization error in the $L^\infty$ norm for 
$H^2$-conforming FEM are presented in \Cref{sec:FEM}.
The three numerical experiments in \Cref{sec:num} conclude this paper.

\medskip
Standard notation for function spaces applies throughout this paper. Let $C^k(\Omega)$ for $k \in \mathbb{N}$ denote the space of scalar-valued $k$-times continuously differentiable functions. Given a positive parameter $0 < \alpha \leq 1$, the H\"older space $C^{k,\alpha}(\Omega)$ is the subspace of $C^k(\overline{\Omega})$ such that all partial derivates of order $k$ are H\"older continuous with exponent $\alpha$.
For any set $\omega \subset \mathbb{R}^n$, $\chi_\omega$ denotes the indicator function associated with $\omega$.
For $A,B \in \mathbb{R}^{n \times n}$, the Euclidean scalar product $A:B \coloneqq \sum_{j,k = 1}^{n} A_{jk} B_{jk}$ induces the Frobenius norm $|A| \coloneqq \sqrt{A:A}$ in $\mathbb{R}^{n \times n}$. 
The notation $|\cdot|$ also denotes the absolute value of a scalar
or the length of a vector.
The relation $A \leq B$ of symmetric matrices $A,B \in \mathbb{S}$ holds 
whenever $B - A$ is positive semidefinite.

\section{Stability estimate}
\label{sec:HJB}
We first recall the concept of viscosity solutions to the HJB equation \eqref{def:HJB}.
\begin{definition}[viscosity solution]
	Let $f \in C(\Omega)$ and $0 \leq \varepsilon \leq 1/n$ be given.
	A function $v \in C(\overline{\Omega})$ is a viscosity subsolution 
	(resp.\ supersolution) to $F_\varepsilon(f;x,\D^2 v) = 0$ if, 
	for all $x_0 \in \Omega$ and $\varphi \in C^2(\Omega)$ such that 
	$v - \varphi$ has a local maximum (resp.\ minimum) at $x_0$, 
	$F_\varepsilon(f;x,\D^2 \varphi) \leq 0$ (resp.\ $F_\varepsilon(f;x,\D^2 \varphi) \geq 0$).
	If $v$ is viscosity sub- and supersolution, then $v$ is called 
	viscosity solution to $F_\varepsilon(f;x,\D^2 v) = 0$.
\end{definition}
The following result provides the first tool in the analysis of this section.
\begin{lemma}[classical comparison principle]\label{lem:comparison-principle}
	Given $0 \leq \varepsilon \leq 1/n$ and a continuous right-hand side $f \in C(\Omega)$,
	where we assume $f\geq0$ if $\varepsilon=0$,
	let $v^* \in C(\overline{\Omega})$ resp.\
	$v_* \in C(\overline{\Omega})$ be a super- resp.~subsolution to the PDE
	\begin{align}\label{def:HJB-no-Dirichlet-data}
		F_\varepsilon(f;x,\D^2 v) = 0 \text{ in } \Omega.
	\end{align}
	If $v_* \leq v^*$ on $\partial \Omega$, then $v_* \leq v^*$ in $\overline{\Omega}$.
\end{lemma}
\begin{proof}
	The proof applies the arguments from \cite[Section 3]{CrandallIshiiLions1992} to the PDE \eqref{def:HJB-no-Dirichlet-data} and can follow \cite[Lemma 3.6]{FengJensen2017} with straightforward modifications; further details are therefore omitted.
\end{proof}
An extended version of \Cref{lem:comparison-principle} below is the following.
\begin{lemma}[comparison principle]\label{lem:extended-comparison-principle}
	Given any $0 \leq \varepsilon_* \leq \varepsilon^* \leq 1/n$ and $f_*,f^* \in C(\Omega)$ with $f_* \leq f^*$ in $\Omega$,
	where we assume $f_*\geq0$ if $\varepsilon_*=0$,
	let $v_*, v^* \in C(\overline{\Omega})$ be viscosity solutions to
	\begin{align*}
		F_{\varepsilon^*}(f_*; x, \D^2 v^*) = 0 \text{ in } \Omega \quad\text{and}\quad
		F_{\varepsilon_*}(f^*; x, \D^2 v_*) = 0 \text{ in } \Omega.
	\end{align*}
	If $v_* \leq v^*$ on $\partial \Omega$, then $v_* \leq v^*$ in $\overline{\Omega}$.
\end{lemma}
\begin{proof}
	Given any test function $\varphi \in C^2(\Omega)$ and $x \in \Omega$ such that
	$v^* - \varphi$ has a local minimum at $x$,
	then $F_{\varepsilon^*}(f_*; x, \D^2 v^*) = 0$ in the sense of viscosity solutions implies
	$0 \leq F_{\varepsilon^*}(f_*; x, \D^2 \varphi(x))$.
	This, $f_* \leq f^*$ in $\Omega$, and $\mathbb{S}(\varepsilon^*) \subset \mathbb{S}(\varepsilon_*)$ show
	\begin{align}
		0 \leq F_{\varepsilon^*}(f_*; x, \D^2 \varphi(x)) \leq F_{\varepsilon_*}(f^*; x, \D^2 \varphi(x)),
	\end{align}
	whence $v^*$ is viscosity supersolution to the PDE $F_{\varepsilon_*}(f^*; x, \D^2 v_*) = 0$.
	Therefore, the comparison principle from \Cref{lem:comparison-principle} with $v_* \leq v^*$ on $\partial \Omega$ concludes $v_* \leq v^*$ in $\overline{\Omega}$.
\end{proof}
The comparison principle from \Cref{lem:comparison-principle} allows for the existence and uniqueness of viscosity solutions \eqref{def:HJB} by Perron's method.
\begin{proposition}[properties of HJB equation]
	\label{prop:properties-HJB}
	Given any $0 \leq \varepsilon \leq 1/n$, $f \in C(\Omega) \cap L^n(\Omega)$, 
	where we assume $f\geq0$ if $\varepsilon=0$,
	and $g \in C(\partial\Omega)$, there exists a unique viscosity solution $u \in C(\overline{\Omega})$ to the HJB equation \eqref{def:HJB}.
	It satisfies (a)--(b):
	\begin{enumerate}[wide]
		\item[(a)] (viscosity = Alexandrov) If $\varepsilon = 0$ and $f \geq 0$ is nonnegative, then the viscosity solution to the HJB equation \eqref{def:HJB} and the Alexandrov solution to the Monge--Amp\`ere equation \eqref{def:Monge-Ampere} coincide.
		\item[(b)] (interior regularity for HJB) If $\varepsilon > 0$ and $f \in C^{0,\alpha}(\Omega)$ with $0 < \alpha < 1$, then $u \in C(\overline{\Omega}) \cap C^{2,\kappa}_{\text{loc}}(\Omega)$ with
		a constant $0 < \kappa < 1$ that solely depends on $\alpha$ and $\varepsilon$.
		\item[(c)] (interior regularity for Monge--Amp\`ere) If $\varepsilon = 0$, $f \in C^{0,\alpha}(\Omega)$ with $0 < \alpha < 1$, $f > 0$ in $\overline{\Omega}$, and $g \in C^{1,\beta}(\partial \Omega)$ with $\beta > 1 - 2/n$, then $u \in C(\overline{\Omega}) \cap C^{2,\alpha}_{\text{loc}}(\Omega)$.
	\end{enumerate}
\end{proposition}
\begin{proof}
	On the one hand, an elementary reasoning as in the proof of
	\Cref{lem:extended-comparison-principle} proves that the viscosity 
	solution $v^*$ to the Poisson equation 
	$F_{\varepsilon^*}(f_*; x, \D^2 v^*) = 0$ with $\varepsilon^* \coloneqq 1/n$, 
	$f_* \coloneqq f$, and Dirichlet data $v^* = g$ on $\partial \Omega$ 
	is a viscosity supersolution to \eqref{def:HJB}.
	On the other hand, the Alexandrov solution $v_*$ to the Monge--Amp\`ere equation
	\eqref{def:Monge-Ampere} with the right-hand side $|f|$ \cite[Theorem 2.14]{Figalli2017} is the viscosity solution to 
	the HJB equation $F_{\varepsilon_*}(f^*; x, \D^2 v_*) = 0$ 
	with $\varepsilon_* \coloneqq 0$, $f^* \coloneqq |f|$, 
	and Dirichlet data $v_* = g$ on $\partial\Omega$ \cite[Proposition 1.3.4]{Gutierrez2016}.
	Hence, the function $v_*$ is viscosity subsolution to \eqref{def:HJB}.
	Therefore, Perron's method \cite[Theorem 4.1]{CrandallIshiiLions1992} and the comparison principle from \Cref{lem:comparison-principle} conclude the existence and uniqueness of viscosity solutions to \eqref{def:HJB}.
	The combination of \cite[Theorem 3.3 and Theorem 3.5]{FengJensen2017} with \cite[Proposition 1.3.4]{Gutierrez2016} implies the assertion in (a).
	The interior regularity in (b) is a classical result from \cite{CaffarelliCabre1995,Safonov1988}.
	For the Monge--Amp\`ere equation, the interior regularity in (c) holds under the assumption that the Alexandrov solution $u$ is strictly convex \cite[Corollary 4.43]{Figalli2017}.
	Sufficient conditions for this are that $f > 0$ is bounded away from zero and $g \in C^{1,\beta}(\partial \Omega)$ is sufficiently smooth \cite[Corollary 4.11]{Figalli2017}.
\end{proof}
Some comments are in order, before we state a precise version of the $L^\infty$ stability estimate \eqref{ineq:stability-intro} from the introduction.
In general, these estimates arise from the Alexandrov--Bakelman--Pucci 
maximum principle for the uniform elliptic Pucci operator, cf.~\cite{CaffarelliCrandallKocanSwiech1996} and the references therein for further details.
However, the constant therein may depend on the ellipticity constant of $F_\varepsilon$ and therefore, on $\varepsilon$. 
In the case of the HJB equation \eqref{def:HJB} that approximates the Monge--Amp\`ere equation \eqref{def:Monge-Ampere} as $\varepsilon \to 0$, the Alexandrov maximum principle is the key argument to avoid a dependency on $\varepsilon$.
Recall the constant $c_n$ from \Cref{lem:Alexandrov-maximum}.

\begin{theorem}[$L^\infty$ stability]\label{lem:stability-regularized-PDE-L-infty}
	Given a nonnegative parameter $0 \leq \varepsilon \leq 1/n$ and right-hand sides $f_1, f_2 \in C(\overline{\Omega})$,
	where we assume $f_1,f_2\geq0$ if $\varepsilon=0$,
	let $v_1, v_2 \in C(\overline{\Omega})$ be viscosity solutions to the HJB equation $F_\varepsilon(f_j; x, \D^2 v_j) = 0 \text{ in } \Omega$
	for $j \in \{1,2\}$. Then, for any subset $\omega \subset \Omega$,
	\begin{align}
		\|v_1 - v_2\|_{L^\infty(\omega)} \leq \|v_1 - v_2\|_{L^\infty(\partial \Omega)} + \frac{C}{n} \max_{x \in \overline{\omega}} \mathrm{dist}(x,\partial \Omega)^{1/n}\|f_1-f_2\|_{L^n(\Omega)}
		\label{ineq:stability-estimate}
	\end{align}
	with the constant $C \coloneqq c_n \mathrm{diam}(\Omega)^{(n-1)/n}$.
	In particular,
	\begin{align}
		\|v_1 - v_2\|_{L^\infty(\Omega)} 
		\leq \|v_1 - v_2\|_{L^\infty(\partial \Omega)} 
		 + \frac{C}{n} (\mathrm{diam}(\Omega)/2)^{1/n}\|f_1-f_2\|_{L^n(\Omega)}.
		\label{ineq:stability-estimate-global}
	\end{align}
\end{theorem}
\begin{proof}
	The proof is divided into two steps.\medskip
	
	\noindent\emph{Step 1:} The first step establishes \eqref{ineq:stability-estimate} under the assumptions $f_2 \leq f_1$ in $\overline{\Omega}$ and $v_1 \leq v_2$ on $\partial \Omega$.
	For $f_\Delta \coloneqq f_1 - f_2 \geq 0$, let the sequence $(f_{\Delta,k})_{k \in \mathbb{N}}$ of smooth functions $f_{\Delta,k} \in C^\infty(\overline{\Omega})$ approximate $f_\Delta \in C(\overline{\Omega})$ from above such that $f_\Delta \leq f_{\Delta,k}$ and $0 < f_{\Delta,k}$ in $\overline{\Omega}$ for all $k \in \mathbb{N}$ and $\lim_{k \to \infty} \|f_k - f_{\Delta,k}\|_{L^\infty(\Omega)} = 0$.
	Let $w_k \in C(\overline{\Omega})$ be viscosity solutions to the PDE, for all $k \in \mathbb{N}$,
	\begin{align}
		F_\varepsilon(f_{\Delta,k}; x, \D^2 w_k) = 0 \text{ in } \Omega \quad&\text{and}\quad w_k = 0 \text{ on } \partial \Omega.
		\label{def:stability-HJB-w-k}
	\end{align}
	Since $v_1 \leq v_2$ on $\partial \Omega$ and $f_2 \leq f_1$ by assumption of Step 1, \Cref{lem:extended-comparison-principle} proves
	\begin{align}\label{ineq:proof-stability-HJB-upper-bound}
		v_1 \leq v_2 \text{ in } \overline{\Omega}.
	\end{align}
	\Cref{prop:properties-HJB}(b)--(c)
	provides the interior regularity $w_k \in C^{2,\alpha}_\mathrm{loc}(\Omega)$ for some positive parameter $\alpha$ that (possibly) depends on $\varepsilon$.
	In particular, $w_k \in C^2(\Omega)$ is a classical solution to the PDE \eqref{def:stability-HJB-w-k}.
	We define the continuous function
	$v_* \coloneqq v_2 - \|v_1 - v_2\|_{L^\infty(\partial \Omega)} + w_k \in C(\overline{\Omega})$.
	Given any $x \in \Omega$ and $\varphi \in C^2(\Omega)$ such that $v_* - \varphi = v_2 - (\|v_1 - v_2\|_{L^\infty(\partial \Omega)} - w_k + \varphi)$ has a local maximum at $x$, the function $\psi \coloneqq \|v_1 - v_2\|_{L^\infty(\partial \Omega)} - w_k + \varphi \in C^2(\Omega)$ is smooth and, therefore, an admissible test function in the definition of viscosity solutions.
	Since $v_2$ is viscosity solution to $F_\varepsilon(f_2; x, \D^2 v_2) = 0$, 
	$F_\varepsilon(f_2; x, \D^2 \psi(x)) \leq 0$ follows.
	This, $\D^2 \psi = \D^2(\varphi - w_k)$, the sub-additivity $\sup(X + Y) \leq \sup X + \sup Y$ of the supremum, $f_\Delta \leq f_{\Delta,k}$, and $F_\varepsilon(f_{\Delta,k}; x, \D^2 w_k(x)) = 0$ from \eqref{def:stability-HJB-w-k} lead to
	\begin{align*}
		F_\varepsilon(f_1; x, \D^2 \varphi(x)) &\leq F_\varepsilon(f_2; x, \D^2 \psi(x)) + F_\varepsilon(f_\Delta; x, \D^2 w_k(x))\\
		&\leq F_\varepsilon(f_2; x, \D^2 \psi(x)) + F_\varepsilon(f_{\Delta,k}; x, \D^2 w_k(x)) \leq 0,
	\end{align*}
	whence $v_*$ is viscosity subsolution to the PDE $F_\varepsilon(f_1; x, \D^2 v) = 0$ in $\Omega$.
	Therefore, $v_* \leq v_1$ on $\partial \Omega$ by design and the comparison principle from \Cref{lem:comparison-principle} provide
	\begin{align}\label{ineq:proof-stability-HJB-lower-bound}
		v_* \leq v_1 \text{ in } \overline{\Omega}.
	\end{align}
	On the one hand, the zero function with $F_\varepsilon(f_{\Delta,k}; x, 0) \geq 0$ is a viscosity supersolution to $F_\varepsilon(f_{\Delta,k}; x, \D^2 w_k) = 0$.
	Hence, the comparison principle from \Cref{lem:comparison-principle} shows $w_k \leq 0$ in $\overline{\Omega}$.
	On the other hand, \Cref{prop:properties-HJB}(a) proves that the Alexandrov solution $z_k \in C(\overline{\Omega})$ to $\det \D^2 z_k = (f_{\Delta,k}/n)^n$ with homogenous boundary is
	viscosity solution to $F_0(f_{\Delta,k}; x, \D^2 z_k) = 0$ and \Cref{lem:extended-comparison-principle} reveals $z_k \leq w_k$, whence $z_k \leq w_k \leq 0$ in $\overline{\Omega}$.
	Consequently, the Alexandrov maximum principle from \Cref{lem:Alexandrov-maximum} and 
	$\mathcal L^n(\partial z_k(\Omega))^{1/n} = \|(f_{\Delta,k}/n)^n\|_{L^1(\Omega)}^{1/n} = \|f_{\Delta,k}\|_{L^n(\Omega)}/n$ imply
	\begin{align}
		0 \leq - w_k \leq - z_k \leq \frac{C}{n}\max_{x \in \overline{\omega}} \mathrm{dist}(x,\partial \Omega)^{1/n}\|f_{\Delta,k}\|_{L^n(\Omega)} \quad\text{in } \overline{\omega}
		\label{ineq:proof-stability-HJB-bound-w-k}
	\end{align}
	for any subset $\omega \subset \Omega$.
	The combination of \eqref{ineq:proof-stability-HJB-upper-bound}--\eqref{ineq:proof-stability-HJB-bound-w-k}
	with $v_* = v_2 - \|v_1 - v_2\|_{L^\infty(\partial \Omega)} + w_k$
	results in
	\begin{align*}
		\|v_1 - v_2\|_{L^\infty(\omega)} \leq \|v_2 - v_*\|_{L^\infty(\omega)} = \|v_1 - v_2\|_{L^\infty(\partial \Omega)} + \|w_k\|_{L^\infty(\omega)}&\\
		\leq \|v_1 - v_2\|_{L^\infty(\partial \Omega)} + \frac{C}{n} \max_{x \in \overline{\omega}} \mathrm{dist}(x,\partial \Omega)^{1/n}\|f_{\Delta,k}\|_{L^n(\Omega)}&.
	\end{align*}
	A passage of the right-hand side to the limit as $k \to \infty$ and $\lim_{k \to \infty} \|f_{\Delta,k}\|_{L^n(\Omega)} = \|f_\Delta\|_{L^n(\Omega)}$ conclude \eqref{ineq:stability-estimate}.\medskip
	
	\noindent\emph{Step 2:} The second step establishes \eqref{ineq:stability-estimate} without the additional assumptions from Step 1.
	For the functions $f_* \coloneqq \min\{f_1,f_2\}$, $f^* \coloneqq \max\{f_1,f_2\}$, and
	$f_\Delta \coloneqq f^* - f_* = |f_1 - f_2| \geq 0$,
	let $v^*,v_* \in C(\overline{\Omega})$ be viscosity solutions to the PDE
	\begin{align}
		F_\varepsilon(f_*; x, \D^2 v^*) = 0 \text{ in } \Omega \quad&\text{and}\quad v^* = \max\{v_1,v_2\} \text{ on } \partial \Omega,
		\label{def:stability-HJB-v^*}\\
		F_\varepsilon(f^*; x, \D^2 v_*) = 0 \text{ in } \Omega \quad&\text{and}\quad v_* = \min\{v_1,v_2\} \text{ on } \partial \Omega,
		\label{def:stability-HJB-v_*}
	\end{align}
	Since $f_* \leq f_j \leq f^*$ and $v_* \leq v_j \leq v^*$ on $\partial \Omega$ for $j \in \{1,2\}$, \Cref{lem:extended-comparison-principle} verifies $v_* \leq \{v_1,v_2\} \leq v^*$ in $\overline{\Omega}$, whence
	\begin{align}
		\|v_1 - v_2\|_{L^\infty(\omega)} \leq \|v^* - v_*\|_{L^\infty(\omega)}\quad\text{for any open subset } \omega \subset \Omega.
		\label{ineq:proof-stability-step-2-v1-v2}
	\end{align}
	The application of Step 1 to the viscosity solutions $v^*,v_*$ of \eqref{def:stability-HJB-v^*}--\eqref{def:stability-HJB-v_*} with $f_* \leq f^*$ and $v_* \leq v^*$ on $\partial \Omega$, and the identity $\max\{a,b\} - \min\{a,b\} = |a - b|$ reveal
	\begin{align*}
		\|v^* - v_*\|_{L^\infty(\omega)} \leq \|v_1 - v_2\|_{L^\infty(\partial \Omega)} + \frac{C}{n} \max_{x \in \overline{\omega}}\mathrm{dist}(x,\partial \Omega)^{1/n}\|f_1 - f_2\|_{L^n(\Omega)}.
	\end{align*}
	The combination of this with \eqref{ineq:proof-stability-step-2-v1-v2} concludes \eqref{ineq:stability-estimate}.
\end{proof}

The stability estimate from \Cref{lem:stability-regularized-PDE-L-infty} motivates a solution concept for the HJB equation \eqref{def:HJB} with $L^n$ right-hand sides.
\begin{lemma}[generalized viscosity solution]\label{lem:generalized-viscosity-solution}
	Given $f \in L^n(\Omega)$, $g \in C(\partial\Omega)$ and $0 \leq \varepsilon \leq 1/n$,
	where we assume $f\geq0$ if $\varepsilon=0$,
	there exists a unique function $u \in C(\overline{\Omega})$ such that $u$ is the uniform limit of any sequence $(u_j)_{j \in \mathbb{N}}$ of viscosity solutions $u_j \in C(\overline{\Omega})$ to
	\begin{align}\label{def:generalized-solution-sequence}
		F_\varepsilon(f_j; x, \D^2 u_j) = 0 \text{ in } \Omega \quad\text{and}\quad u_j = g_j \text{ on } \partial \Omega
	\end{align}
	for right-hand sides $f_j \in C(\overline{\Omega})$ and Dirichlet data $g_j \in C(\overline{\Omega})$ with $\lim_{j \to \infty} \|f - f_j\|_{L^n(\Omega)} = 0$ and $\lim_{j \to \infty} \|g - g_j\|_{L^\infty(\partial \Omega)} = 0$.
	The function $u$ is called generalized viscosity solution to \eqref{def:HJB}.
	If $\varepsilon = 0$ and $f \geq 0$, then the generalized viscosity solution to \eqref{def:HJB} and the Alexandrov solution to \eqref{def:Monge-Ampere} coincide.
\end{lemma}
\begin{proof}
	Let $(f_j)_{j \in \mathbb{N}} \subset C(\overline{\Omega})$ (resp.~$(g_j)_{j \in \mathbb{N}} \subset C(\overline{\Omega})$) approximate $f$ in $L^n(\Omega)$ (resp.~$g$ in $C(\partial \Omega)$). For any index $j,k \in \mathbb{N}$, the stability estimate \eqref{ineq:stability-estimate-global} from \Cref{lem:stability-regularized-PDE-L-infty} provides
	\begin{align*}
		\|u_j - u_k\|_{L^\infty(\Omega)} 
		\leq \|g_j - g_k\|_{L^\infty(\partial \Omega)} + \frac{C}{n} (\mathrm{diam}(\Omega)/2)^{1/n}\|f_j - f_k\|_{L^n(\Omega)}.
	\end{align*}
	Since $(f_j)_{j \in \mathbb{N}}$ (resp.~$(g_j)_{j \in \mathbb{N}}$) is a Cauchy sequence in $L^n(\Omega)$ (resp.~$C(\partial \Omega)$),
	this implies that $(u_j)_{j \in \mathbb{N}}$ is a Cauchy sequence 
	in the Banach space $C(\overline{\Omega})$ endowed with the $L^\infty$ norm.
	Therefore, there exists $u \in C(\overline{\Omega})$ with 
	$\lim_{j \to \infty} \|u - u_j\|_{L^\infty(\Omega)} = 0$.
	It remains to prove that $u$ is independent of the choice of the approximation sequences for $f$ and $g$.
	To this end, let $(\widetilde{f}_j)_{j \in \mathbb{N}}$ be another sequence of continuous functions $\widetilde{f}_j \in C(\overline{\Omega})$ with $\lim_{j \to \infty} \|f - \widetilde{f}_j\|_{L^n(\Omega)} = 0$.
	Then the sequence $(\widetilde{u}_j)_{j \in \mathbb{N}}$ of viscosity solutions $\widetilde{u}_j \in C(\overline{\Omega})$ to \eqref{def:generalized-solution-sequence} with $f_j$ replaced by $\widetilde{f}_j$ converges uniformly to some $\widetilde{u} \in C(\overline{\Omega})$. 
	The stability estimate \eqref{ineq:stability-estimate-global}
	from \Cref{lem:stability-regularized-PDE-L-infty} shows
	\begin{align*}
		\|u_j - \widetilde{u}_j\|_{L^\infty(\Omega)}
		\leq \frac{C}{n}(\mathrm{diam}(\Omega)/2)^{1/n}\|f_j - \widetilde{f}_j\|_{L^n(\Omega)}
	\end{align*}
	for any $j \in \mathbb{N}$.
	The right-hand side of this vanishes in the limit and the left-hand side converges to $\|u - \widetilde{u}\|_{L^\infty(\Omega)}$ as $j \to \infty$, whence $u = \widetilde{u}$ in $\overline{\Omega}$.
	If $f \geq 0$, then there exists a sequence $(f_j)_{j \in \mathbb{N}}$ of nonnegative continuous functions $0 \leq f_j \in C(\overline{\Omega})$ with $\lim_{j \to \infty} \|f - f_j\|_{L^\infty(\Omega)}$ (e.g., from convolution with a nonnegative mollifier). 
	\Cref{prop:properties-HJB}(a) provides, for all $j \in \mathbb{N}$, that the viscosity solution $u_j$ to \eqref{def:generalized-solution-sequence} with $\varepsilon = 0$ is the Alexandrov solution to $\det \D^2 u_j = f_j$ in $\Omega$.
	Since $u_j$ converges uniformly to the generalized viscosity solution $u$ to \eqref{def:HJB}, the stability of Alexandrov solutions \cite[Corollary 2.12 and Proposition 2.16]{Figalli2017} concludes that $u$ is the Alexandrov solution to \eqref{def:Monge-Ampere}.
\end{proof}
By approximation of the right-hand sides, the stability estimates from \Cref{lem:stability-regularized-PDE-L-infty} also applies to generalized viscosity solutions to the HJB equation \eqref{def:HJB}.
\begin{corollary}[extended $L^\infty$ stability]\label{lem:extended-stability}
	Given any $0 \leq \varepsilon \leq 1/n$, $f_j \in L^n(\Omega)$,
	where we assume $f_j\geq0$ if $\varepsilon=0$,
	and $g_j \in C(\overline{\Omega})$, the generalized viscosity solutions $v_j \in C(\overline{\Omega})$ to $F_\varepsilon(f_j;x,\D^2 v_j) = 0$ in $\Omega$ for $j \in \{1,2\}$ satisfy \eqref{ineq:stability-estimate}--\eqref{ineq:stability-estimate-global}.
\end{corollary}
\begin{proof}
	For any index $j \in \{1,2\}$, there exists a sequence $(f_{j,k})_{j \in \mathbb{N}}$ of smooth functions $f_{j,k} \in C^\infty(\overline{\Omega})$ that approximates $f_j$ in $L^n(\Omega)$, i.e., $\lim_{k \to \infty} \|f_j - f_{j,k}\|_{L^n(\Omega)} = 0$.
	Given any $j \in \{1,2\}$ and $k \in \mathbb{N}$, let $v_{j,k} \in C(\overline{\Omega})$ denote the viscosity solution to the HJB equation
	$F_\varepsilon(f_{j,k}; x, \D^2 v_{j,k}) = 0$ in $\Omega$ and $v_{j,k} = v_j$ on $\partial \Omega$.
	The $L^\infty$ stability estimate \eqref{ineq:stability-estimate} from \Cref{lem:stability-regularized-PDE-L-infty} shows, for any $k \in \mathbb{N}$, that
	\begin{align*}
		\|v_{1,k} - v_{2,k}\|_{L^\infty(\omega)} &\leq \|v_1 - v_2\|_{L^\infty(\partial \Omega)} + \frac{C}{n} \max_{x \in \overline{\omega}} \mathrm{dist}(x,\partial \Omega)^{1/n}\|f_{1,k} - f_{2,k}\|_{L^n(\Omega)}.
	\end{align*}
	The left-hand side of this converges to $\|v_1 - v_2\|_{L^\infty(\Omega)}$ by the definition of generalized viscosity solutions in \Cref{lem:generalized-viscosity-solution}. Hence, $\lim_{k \to \infty} \|f_{1,k} - f_{2,k}\|_{L^n(\Omega)} = \|f_1 - f_2\|_{L^n(\Omega)}$ concludes the proof.
\end{proof}
\begin{remark}[$L^\infty$ stability for Alexandrov solutions]\label{rem:stability-Alexandrov-solution}	
	If the right-hand sides $0 \leq f_1,f_2 \in L^n(\Omega)$ are nonnegative,
	then the generalized solutions $v_1,v_2$ from \Cref{lem:extended-stability} are Alexandrov solutions to $\det \D^2 v_j = (f_j/n)^n$, cf.~\Cref{lem:generalized-viscosity-solution}.
	Therefore, \Cref{lem:extended-stability} provides $L^\infty$ stability estimates for Alexandrov solutions.
\end{remark}

The convexity of the differential operator $F_\varepsilon$ in $\mathbb{S}$ leads to existence (and uniqueness) of strong solutions $u_\varepsilon \in C(\overline{\Omega}) \cap W^{2,n}_\mathrm{loc}(\Omega)$ to \eqref{def:HJB} for any $\varepsilon > 0$, $f \in L^n(\Omega)$, and $g \in C(\partial \Omega)$ \cite{CaffarelliCrandallKocanSwiech1996}. It turns out that strong solutions are generalized viscosity solutions.
For the purpose of this paper, we only provide a weaker result.

\begin{theorem}[strong solution implies generalized viscosity solution]\label{thm:strongimpliesviscosity}
	Let $0 < \varepsilon \leq 1/n$, $f \in L^n(\Omega)$, 
	and $g \in C(\partial\Omega)$ be given. Suppose that $u_\varepsilon \in W^{2,n}(\Omega)$
	is a strong solution to \eqref{def:HJB} in the sense that \eqref{def:HJB} is satisfied a.e.~in $\Omega$.
	Then this strong solution $u_\varepsilon$ is the unique generalized 
	viscosity solution to \eqref{def:HJB}.
\end{theorem}
The proof of \Cref{thm:strongimpliesviscosity} utilizes the following elementary result.
\begin{lemma}[computation and stability of right-hand side]\label{lem:computation-stability-RHS}
	Let $\varepsilon > 0$ be given.
	For any $M \in \mathbb{S}$, there exists a unique $\xi(M) \in \mathbb{R}$ such that $\max_{A \in \mathbb{S}(\varepsilon)}(- A: M + \xi(M)\sqrt[n]{\det A}) = 0$.
	Furthermore, any $M, N \in \mathbb{S}$ satisfy the stability $|\xi(M) - \xi(N)| \leq C(\varepsilon)|M - N|$ with a constant depending on the regularization parameter $\varepsilon$.
\end{lemma}
\begin{proof}
	Given a symmetric matrix $M \in \mathbb{S}$, define the continuous real-valued function
	\begin{align}
		\Psi_{M}(\xi) \coloneqq \max_{A \in \mathbb{S}(\varepsilon)} (-A:M + \xi\sqrt[n]{\det A}).
	\end{align}
	Since $\Psi_M$ is strictly monotonically increasing with the limits $\lim_{\xi \to -\infty} \Psi_M = -\infty$ and $\lim_{\xi \to \infty} \Psi_M = +\infty$, there exists a unique root $\xi(M)$ such that $\Psi_M(\xi(M)) = 0$.
	For any $M,N \in \mathbb{S}$, the inequality $\max X - \max Y \leq \max (X - Y)$ shows
	\begin{align}\label{ineq:proof-uniqueness-rhs-HJB}
		0 = \Psi_{M}(\xi(M)) - \Psi_{N}(\xi(N)) \leq \Psi_{M - N}(\xi(M) - \xi(N)).
	\end{align}
	Let $A \in \mathbb{S}(\varepsilon)$ be chosen such that $\Psi_{M - N}(\xi(M) - \xi(N)) = -A:(M-N) + (\xi(M) - \xi(N))\sqrt[n]{\det A}$. Then it follows from \eqref{ineq:proof-uniqueness-rhs-HJB} that
	\begin{align}
		\xi(N) - \xi(M) \leq A:(N-M)/\sqrt[n]{\det A}\leq |A||M-N|/\sqrt[n]{\det A}.
		\label{ineq:proof-uniqueness-rhs-HJB-one-side}
	\end{align}
	Exchanging the roles of $M$ and $N$ in \eqref{ineq:proof-uniqueness-rhs-HJB-one-side} leads to 
	$\xi(M) - \xi(N) \leq |B||M-N|/\sqrt[n]{\det B}$ for some $B \in \mathbb{S}(\varepsilon)$. Since $|A|/\sqrt[n]{\det A} \leq 1/(\sqrt{\varepsilon^{n-1}(1-(n-1)\varepsilon)})$ holds for any $A \in \mathbb{S}(\varepsilon)$, the combination of this with \eqref{ineq:proof-uniqueness-rhs-HJB-one-side} concludes $|\xi(N) - \xi(M)| \leq |M-N|/\sqrt[n]{\varepsilon^{n-1}(1-(n-1)\varepsilon)}$.
\end{proof}
\begin{proof}[Proof of \Cref{thm:strongimpliesviscosity}]
	Let $v_j \in C^2(\overline{\Omega})$ be a sequence of smooth functions that approximate $u_\varepsilon$ with $\lim_{j \to \infty} \|u_\varepsilon - v_j\|_{W^{2,n}(\Omega)} = 0$.
	\Cref{lem:computation-stability-RHS} proves that there exists
	a (unique) function $f_j \coloneqq \xi(\D^2 v_j)$ with $F_\varepsilon(f_j;x,\D^2 v_j) = 0$ in $\Omega$.
	We apply the stability from \Cref{lem:computation-stability-RHS} twice. First, $|f_j(x) - f_j(y)| \leq C(\varepsilon)|\D^2 v_j(x) - \D^2 v_j(y)|$ for any $x,y \in \Omega$ implies continuity $f_j \in C(\overline{\Omega})$ of $f_j$ and second, $|f(x) - f_j(x)| \leq C(\varepsilon)|\D^2 u_\varepsilon(x) - \D^2 v_j(x)|$ for a.e.~$x \in \Omega$ implies the convergence $\lim_{j \to \infty} \|f - f_j\|_{L^n(\Omega)} = 0$.
	Notice from the Sobolev embedding that $v_j$ converges uniformly to $u_\varepsilon$ in $\overline{\Omega}$ as $j \to \infty$.
	In conclusion, $u_\varepsilon$ is the uniform limit of classical (and in particular, viscosity) solutions $v_j$ such that the corresponding right-hand sides and Dirichlet data converge in the correct norm, i.e.,~$\lim_{j \to \infty} \|f - f_j\|_{L^n(\Omega)} = 0$ and $\lim_{j \to \infty}\|g - v_j\|_{L^\infty(\partial \Omega)} = 0$. \Cref{lem:generalized-viscosity-solution} proves that $u_\varepsilon$ is the unique (generalized) viscosity solution.
\end{proof}

\section{Convergence of the regularization}\label{sec:convergence}
This section establishes the uniform convergence of the generalized viscosity solution $u_\varepsilon$ of the regularized HJB equation \eqref{def:HJB} to the Alexandrov solution $u$ of the Monge--Amp\`ere equation \eqref{def:Monge-Ampere} for any nonnegative right-hand side $0 \leq f \in L^n(\Omega)$.
The proof is carried out in any space dimension $n$ and does not rely on the concept of strong solutions in two space dimensions from \cite{SmearsSueli2013,SmearsSueli2014}.
It departs from a main result of \cite{gt2021}.
\begin{theorem}[convergence of regularization for smooth data]\label{thm:convergence-smooth-data}
	Let $f \in C^{0,\alpha}(\Omega)$,
	$0 < \lambda \leq f \leq \Lambda$,
	and $g \in C^{1,\beta}(\partial\Omega)$
	with positive constants
	$0 < \alpha,\beta < 1$ and $0 < \lambda \leq \Lambda$ be given.
	Let
	$u \in C(\overline{\Omega}) \cap C^{2,\alpha}_{\mathrm{loc}}(\Omega)$
	be the unique classical solution to \eqref{def:Monge-Ampere} from \Cref{prop:properties-HJB}(c).
	\begin{enumerate}[wide]
		\item[(a)] For any sequence
		$0<(\varepsilon_j)_{j \in \mathbb{N}}\leq 1/n$ with $\lim_{j \to \infty} \varepsilon_j = 0$, the sequence $(u_{\varepsilon_j})_{j \in \mathbb{N}}$ of classical solutions $u_{\varepsilon_j} \in C(\overline{\Omega}) \cap C^2(\Omega)$ to \eqref{def:HJB} with $\varepsilon \coloneqq \varepsilon_j$ from \Cref{prop:properties-HJB}(b) converges uniformly to $u$ in $\Omega$ as $j \to \infty$.
		\item[(b)] If $g\equiv 0$, $f \in C^{2,\alpha}(\Omega)$, and
		$f > 0$ in $\overline{\Omega}$,
		then, for some constant $C$ and all $0 < \varepsilon \leq 1/n$, 
		the generalized viscosity solution $u_\varepsilon$
		to \eqref{def:HJB}
		satisfies
		\begin{align*}
			\|u - u_\varepsilon\|_{L^\infty(\Omega)}
			\leq C \varepsilon^{1/(n^2(2n+3))}.
		\end{align*}
	\end{enumerate}
\end{theorem}
\begin{proof}
	The proof of \Cref{thm:convergence-smooth-data} can follow the lines of the proof of \cite[Theorem 4.1]{gt2021}, where \Cref{lem:regularization} below replaces its counterpart \cite[Lemma 4.2]{gt2021} in two space dimensions. We note that the assumption $g \in H^2(\Omega)$ in \cite[Theorem 4.1]{gt2021} is only required for the existence of strong solutions $u_\varepsilon \in H^2(\Omega)$ and can be dropped. Further details of the proof are omitted.
\end{proof}
\begin{lemma}[effect of regularization]\label{lem:regularization}
	Given $0 < \varepsilon \leq 1/n$, $M \in \mathbb{S}$, and $\xi > 0$, suppose that $|M|_n^n \leq \xi^n(1/\varepsilon - (n-1))/n^n$ and $\max_{A \in \mathbb{S}(0)} (-A:M + \xi\sqrt{\det A}) = 0$, then $\max_{A \in \mathbb{S}(\varepsilon)} (-A:M + \xi\sqrt{\det A}) = 0$.
\end{lemma}
\begin{proof}
	The assumption $\max_{A \in \mathbb{S}(0)} (-A:M + \xi\sqrt{\det A}) = 0$ implies that $M > 0$ is positive definite and $\det M = (\xi/n)^n$ \cite[p.~51]{Krylov1987}.
	Let $\varrho_1,\dots,\varrho_n$ denote the positive eigenvalues of $M$
	and $t_j \coloneqq \varrho_j^{-1}/(\sum_{k=1}^n \varrho_k^{-1})$
	for $j = 1,\dots,n$. By design of $t_j$,
	\begin{align*}
		\varrho_j^{-1} 
		 = t_j\left(
		     \frac{\varrho_1^{-1} \dots \varrho_n^{-1}}{t_1 \dots t_n}
		       \right)^{1/n},
	\end{align*}
	whence $\varrho_j = \xi(t_1 \dots t_n)^{1/n}/(n t_j)$.
	Without loss of generality, suppose that $t_1 \leq t_2 \leq \dots \leq t_n$.
	The elementary bound $t_1 \dots t_n \geq t_1^{n-1}(1-(n-1)t_1)$ 
	proves
	\begin{align*}
		\xi^n(1-(n-1)t_1)/t_1 \leq \xi^n(t_1 \dots t_n)/(nt_1)^n 
		=
		n^n \varrho_1^n 
		\leq n^n|M|_n^n.
	\end{align*}
	Hence, $1/t_1 \leq n^n|M|_n^n/\xi^n + (n-1) \leq 1/\varepsilon$ by assumption and so, $t_1 \geq \varepsilon$. In particular, $\varepsilon \leq t_1 \leq \dots \leq t_n$ and $t_1 + \dots + t_n = 1$.
	Notice that $t \coloneqq (t_1,\dots,t_n) \in \mathbb{R}^n$ maximizes the scalar-valued function $g : \mathbb{R}^{n} \to \mathbb{R}$ with
	\begin{align*}
		\psi(s) \coloneqq -s_1 \varrho_1 - \dots - s_n \varrho_n + \xi\sqrt[n]{s_1 \dots s_n}
	\end{align*}
	among $s \in S(0)$ with $S(\varepsilon) \coloneqq \{s = (s_1,\dots,s_n) : s \geq \varepsilon \text{ and } s_1 + \dots + s_n = 1\}$.
	Since $\psi(t) = \max_{s \in S(0)} \psi(s) = \max_{A \in \mathbb{S}(0)} (-A:M + \xi\sqrt{\det A})$ \cite[p.~51--52]{Krylov1987} and $t \in S(\varepsilon)$, this implies that $0 = \psi(t) = \max_{A \in \mathbb{S}(\varepsilon)} (-A:M + \xi\sqrt{\det A})$.
\end{proof}
The approximation of nonsmooth data leads to the following convergence result under (almost) minimal assumptions 
(general Borel measures as right-hand sides are excluded).
\begin{theorem}[convergence of regularization]\label{thm:convergence}
	Let a sequence $(\varepsilon_j)_{j \in \mathbb{N}} \subset (0,1/n]$ with $\lim_{j \to \infty} \varepsilon_j = 0$, a nonnegative right-hand side $0 \leq f \in L^n(\Omega)$, and Dirichlet data $g \in C(\partial\Omega)$ be given.
	Then the sequence $(u_j)_{j \in \mathbb{N}}$ of generalized viscosity solutions $u_j \in C(\overline{\Omega})$ to
	\begin{align*}
		F_{\varepsilon_j}(f; x, \D^2 u_j) = 0 \text{ in } \Omega \quad\text{and}\quad u_j = g \text{ on } \partial \Omega
	\end{align*}
	converges uniformly $\lim_{j \to \infty} \|u - u_j\|_{L^\infty(\Omega)} = 0$ to the Alexandrov solution $u$ to the Monge--Amp\`ere equation \eqref{def:Monge-Ampere}.
\end{theorem}
\begin{proof}
	Recall the constant $c_n$ from \Cref{lem:Alexandrov-maximum} and $C \coloneqq c_n \mathrm{diam}(\Omega)^{(n-1)/n}$.
	Given $\delta > 0$, there exist smooth functions $f_\delta, g_\delta \in C^\infty(\overline{\Omega})$ such that
	\begin{enumerate}[wide]
		\item[(i)] $f_\delta > 0$ in $\overline{\Omega}$ and $\|f-f_\delta\|_{L^n(\Omega)} \leq n\delta/(8C(\mathrm{diam}(\Omega)/2)^{1/n})$ (the approximation $f_\delta$ can be constructed by the convolution of $f$ with a nonnegative mollifier plus an additional small constant),
		\item[(ii)] $\|g - g_\delta\|_{L^\infty(\partial \Omega)} \leq \delta/4$.
	\end{enumerate}
	Notice that the bound $f_\delta > 0$ in $\overline{\Omega}$ and the smoothness of the Dirichlet data $g_\delta \in C^{\infty}(\partial \Omega)$ allow for strict convexity of the Alexandrov solution $u_\delta$ to the Monge--Amp\`ere equation $\det \D^2 u_\delta = (f_\delta/n)^n$ with Dirichlet data $u_\delta = g_\delta$ on $\partial \Omega$ \cite[Corollary 4.11]{Figalli2017}.
	This is a crucial assumption in \Cref{thm:convergence-smooth-data}, which leads to the uniform convergence of the sequence $(u_{\delta,j})_{j \in \mathbb{N}}$ of viscosity solutions $u_{\delta,j} \in C(\overline{\Omega})$ to the HJB equation
	\begin{align*}
		F_{\varepsilon_j}(f_\delta; x, \D^2 u_{\delta,j}) = 0 \text{ a.e.~in } \Omega \quad\text{and}\quad u_{\delta,j} = g_\delta \text{ on } \partial \Omega
	\end{align*}
	towards $u_\delta$ as $j \to \infty$. Therefore, there exists a $j_0 \in \mathbb{N}$ such that $\|u_\delta - u_{\delta,j}\|_{L^\infty(\Omega)} \leq \delta/4$ for all $j \geq j_0$.
	The stability estimate \eqref{ineq:stability-estimate-global} from \Cref{lem:extended-stability} and (i)--(ii) provide
	\begin{align*}
		&\|u - u_\delta\|_{L^\infty(\Omega)} + \|u_j - u_{\delta,j}\|_{L^\infty(\Omega)}\\
		&\qquad\leq 2\|g-g_\delta\|_{L^\infty(\partial \Omega)} + \frac{2C}{n} (\mathrm{diam}(\Omega)/2)^{1/n}\|f - f_\delta\|_{L^n(\Omega)} \leq 3\delta/4.
	\end{align*}
	This, the triangle inequality, and $\|u_\delta - u_{\delta,j}\|_{L^\infty(\Omega)} \leq \delta/4$ verify, for all $j \geq j_0$, that $\|u - u_j\|_{L^\infty(\Omega)} \leq \delta$,
	whence $u_j$ converges uniformly to $u$ as $j \to \infty$.\medskip
\end{proof}

\section{A posteriori error estimate}\label{sec:FEM}

In this section we prove an a~posteriori error bound for a given
approximation $v_h$ to the Alexandrov solution $u$ of the 
Monge--Amp\`ere equation. In what follows we assume
a given finite partition $\mathcal T$ of $\overline\Omega$
of closed polytopes
such that the interiors of any distinct $T,K\in\mathcal T$ are disjoint
and the union over $\mathcal T$ equals $\overline\Omega$.
Let $V_h\subset C^{1,1}(\overline\Omega)$ be a subspace of functions in
$C^{2}(T)$ when restricted to any set 
$T\in\mathcal T$ of the partition.
(Here, $C^2$ up to the boundary of $T$ means that there exists a sufficiently
smooth extension of the function $v_h|_{\mathrm{int}(T)}$ to $T$ for $v_h \in V_h$.)
The piecewise Hessian of any $v_h \in V_h$ is denoted by
$\D_\mathrm{pw}^2v_h$.
In practical examples, we think of $V_h$ as a space of
$C^1$-regular finite element functions.
Given any $v \in C(\Omega)$, its convex envelope is defined
as
\begin{align}
	\Gamma_{v} (x) \coloneqq
	\sup_{\substack{w : \mathbb{R}^n \to \mathbb R
		\text{ affine}\\
		w\leq v}} w (x)
	\quad\text{for any }x\in \Omega.
\end{align}
Let $\mathcal C_v \coloneqq \{x\in \Omega: v (x) = \Gamma_{v}(x)\}$ denote the contact set of $v$.
\begin{theorem}[guaranteed error control for Monge--Amp\`ere]
	\label{thm:GUB-MA}
	Given a nonnegative right-hand side $f \in L^n(\Omega)$ and $g \in C(\partial\Omega)$,
	let $u \in C(\overline{\Omega})$ be the Alexandrov solution to \eqref{def:Monge-Ampere}.
	Let $v_h\in V_h$ with its convex envelope $\Gamma_{v_h}$ be given
	and define $f_h \coloneqq \chi_{\mathcal{C}_{v_h}} n  (\det \D^2_\mathrm{pw} v_h)^{1/n}$.
	For any convex subset $\Omega' \subset \Omega$, we have
	\begin{align}
		\|u - \Gamma_{v_h}\|_{L^\infty(\Omega)} 
		\leq \limsup_{x \to \partial \Omega} |(g - \Gamma_{v_h})(x)|
		+ \frac{c_n}{2^{1/n}n}\mathrm{diam}(\Omega')\|f - f_h\|_{L^n(\Omega')}&\nonumber\\
		 + \frac{c_n}{n}\mathrm{diam}(\Omega)^{(n-1)/n} \max_{x\in\overline{\Omega\setminus\Omega'}}
		\operatorname{dist}(x,\partial\Omega)^{1/n}
		\|f - f_h\|_{L^n(\Omega)}\eqqcolon \mathrm{RHS}_0&.
		\label{ineq:GUB-MA}
	\end{align}
\end{theorem}

The proof of \Cref{thm:GUB-MA} requires the following result on
the Monge--Amp\`ere measure of the convex envelope $\Gamma_{v_h}$.
\begin{lemma}[MA measure of the convex envelope] \label{l:MA4conv}
 The convex envelope $\Gamma_{v_h}$ of any $v_h\in V_h$
  satisfies $\det \D^2 \Gamma_{v_h} = \widetilde{f}_h \d x$ in the sense of 
  Monge--Amp\`ere measure with the nonnegative function
  $  \widetilde f_h \coloneqq \chi_{\mathcal{C}_{v_h}}\det \D_\mathrm{pw}^2 v_h 
  \in L^\infty(\Omega)$.
\end{lemma}
\begin{proof}
	We first claim that $\partial \Gamma_{v_h}(x) = \partial v_h(x) = \{\nabla v_h(x)\}$ holds for all $x \in \Omega \cap \mathcal{C}_{v_h}$. In fact, if $p \in \partial \Gamma_{v_h}(x)$, then $\ell_{x,p}(z) \coloneqq \Gamma_{v_h}(x) + p \cdot (z - x)$ is a supporting hyperplane touching $\Gamma_{v_h}$ from below at $x$.
	By design of the convex envelope $\Gamma_{v_h}$, $\ell_{x,p} \leq v_h$.
	Since $\ell_{x,p}(x) = v_h(x)$ because $x \in \Omega \cap \mathcal{C}_{v_h}$, $\ell_{x,p}$ touches $v_h$ at $x$ from below.
	We deduce $p = \nabla v_h(x)$ from the differentiability of $v_h$. The claim then follows from the fact that the subdifferential $\partial \Gamma_{v_h}$ is nonempty in $\Omega$ \cite[Theorem 23.4]{Rockafellar1970}.
	The set $\partial \Gamma_{v_h}(\Omega \setminus \mathcal{C}_{v_h})$ has 
	Lebesgue measure zero \cite[p.~995]{DePhilippisFigalli2013} and $\partial \Gamma_{v_h}(x) = \partial v_h(x) = \{\nabla v_h(x)\}$ holds for all $x \in \Omega \cap \mathcal{C}_{v_h}$.
	Therefore, the area formula \cite[Theorem A.31]{Figalli2017} implies, for any Borel set $\omega \subset \Omega$, that
	\begin{align*}
		\mu_{\Gamma_{v_h}}(\omega) 
		= \mathcal L^n(\partial \Gamma_{v_h}(\omega))
		= \mathcal L^n(\nabla v_h(\omega \cap \mathcal{C}_{v_h}))
		= \int_{\omega \cap \mathcal{C}_{v_h}} \det \D^2_\mathrm{pw} v_h \d x.
	\end{align*}
	This formula implies that $\chi_{\mathcal{C}_{v_h}} \det \D^2_\mathrm{pw} v_h \geq 0$ is a nonnegative function a.e.~in $\Omega$.
	Consequently,
	$\mu_{\Gamma_{v_h}} = \widetilde f_h \d x$ with 
	$\tilde f_h \coloneqq \chi_{\mathcal{C}_{v_h}} \det \D^2_\mathrm{pw} v_h \geq 0$.
\end{proof}
\begin{proof}[Proof of \Cref{thm:GUB-MA}]
	\Cref{l:MA4conv} proves that the Monge--Amp\`ere measure $\mu_{\Gamma_{v_h}} = (f_h/n)^n \d{x}$ of $\Gamma_{v_h}$ can be expressed by the $L^1$ density function $(f_h/n)^n$.
	In particular, $\Gamma_{v_h}$ is the generalized viscosity solution to $F_0(f_h;x,\D^2 \Gamma_{v_h}) = 0$ in $\Omega$.
	The application of the stability estimate \eqref{ineq:stability-estimate-global} from \Cref{lem:extended-stability} on the convex subset $\Omega' \subset \Omega$ instead of $\Omega$ leads to
	\begin{align*}
		\|u - \Gamma_{v_h}\|_{L^\infty(\Omega')} 
		\leq \|u - \Gamma_{v_h}\|_{L^\infty(\partial \Omega')} 
		     + \frac{c_n}{2^{1/n} n}\mathrm{diam}(\Omega')
		        \|f - f_h\|_{L^n(\Omega')}.
	\end{align*}
	The unknown error $\|u - \Gamma_{v_h}\|_{L^\infty(\partial \Omega')} \leq \|u - \Gamma_{v_h}\|_{L^\infty(\Omega \setminus \Omega')}$ can be bounded by the local estimate \eqref{ineq:stability-estimate} from \Cref{lem:extended-stability} with $\omega \coloneqq \Omega \setminus \Omega'$. If $\Gamma_{v_h} \in C(\overline{\Omega})$ is continuous up to the boundary $\partial \Omega$ of $\Omega$, this reads
	\begin{align*}
		&\|u - \Gamma_{v_h}\|_{L^\infty(\Omega \setminus \Omega')} \leq \|g - \Gamma_{v_h}\|_{L^\infty(\partial \Omega)}\\
		&\qquad +
                   \frac{c_n}{n} \mathrm{diam}(\Omega)^{(n-1)/n} \max_{x\in\overline{\Omega\setminus\Omega'}}
                           \operatorname{dist}(x,\partial\Omega)^{1/n}
		          \|f - f_h\|_{L^n(\Omega)}.
	\end{align*}
	Since $\Gamma_{v_h}$ may only be continuous in the domain $\Omega$, $\|g - \Gamma_{v_h}\|_{L^\infty(\partial \Omega)}$ is replaced by $\limsup_{x \to \partial \Omega} |(g - \Gamma_{v_h})(x)|$ in general. The combination of the two previously displayed formula concludes the proof.
\end{proof}
We note that, for certain examples, the convex envelope $\Gamma_{v_h}$ of an approximation $v_h$ is continuous up to the boundary.
\begin{proposition}[continuity at boundary]\label{prop:continuity-convex-envelope}
	Let $v \in C^{0,1}(\overline{\Omega})$ be Lipschitz
	continuous such that $v|_{\partial \Omega}$ can be extended to a Lipschitz-continuous convex function $g \in C^{0,1}(\overline{\Omega})$.
	Then $\Gamma_{v} \in C(\overline{\Omega})$ and $\Gamma_{v} = v$ on $\partial \Omega$.
\end{proposition}
\begin{proof}
	We first prove the assertion for homogenous boundary condition $v|_{\partial \Omega} = 0$.
	Given any point $x \in \Omega$, let $x' \in \partial \Omega$ denote a best approximation of $x$ onto the boundary $\partial \Omega$ so that $|x - x'| = \mathrm{dist}(x,\partial \Omega)$.
	Define the affine function $a_x(z) \coloneqq L(z - x') \cdot (x' - x)/|x - x'|$ for $z \in \Omega$, where $L$ denotes the Lipschitz constant of the function $v \in C^{0,1}(\overline{\Omega})$.
	It is straight-forward to verify that $a_x \leq v$ in $\overline{\Omega}$ \cite[p.~12]{Gutierrez2016}.
	Therefore, $-L\mathrm{dist}(x,\partial \Omega) = a_x(x) \leq \Gamma_{v}(x) \leq 0$ by definition of the convex envelope.
	This shows $\Gamma_{v} \in C(\overline{\Omega})$ with $\Gamma_{v} \equiv 0$ on $\partial \Omega$.
	In the general case, we observe that $v - g \in C^{0,1}(\overline{\Omega})$ is Lipschitz continuous.
	The first case proves $\Gamma_{v - g} \in C(\overline{\Omega})$ with $\Gamma_{v - g} = v - g$ on $\partial \Omega$.
	We deduce that $w \coloneqq g + \Gamma_{v - g} \in C(\overline{\Omega})$ is a convex function with $w \leq v$ in $\Omega$ and $w = v$ on $\partial \Omega$.
	Let $(x_j)_j \subset \Omega$ be a sequence of points converging to some  point $x \in \partial \Omega$ on the boundary.
	For a given $\gamma > 0$, there exists,
	from the uniform continuity of $v - w$ in the compact set $\overline{\Omega}$, a $\delta > 0$ such that $|(v - w)(x_j) - (v - w)(x)| \leq \gamma$ whenever $|x - x_j| \leq \delta$.
	Since $w \leq \Gamma_{v} \leq v$ in $\Omega$, this implies $|(v - \Gamma_{v})(x_j)| \leq \gamma$ for sufficiently large $j$.
	In combination with the triangle inequality and the Lipschitz continuity of $v$, we conclude $|v(x) - \Gamma_{v}(x_j)| \leq \gamma + |v(x) - v(x_j)| \leq \gamma + L|x - x_j|$.
	Therefore, $\lim_{j \to \infty} \Gamma_{v}(x_j) = v(x)$.
\end{proof}

The theory of this paper also allows for an a~posteriori error control for the regularized HJB equation \eqref{def:HJB}.
We state this for the sake of completeness as, in general, it is difficult to quantify the regularization error $\|u - u_\varepsilon\|_{L^\infty(\Omega)}$.
\begin{theorem}[guaranteed $L^\infty$ error control for uniform elliptic HJB]\label{thm:GUB-HJB}
	Given a positive parameter $0 < \varepsilon \leq 1/n$ and a $C^1$ conforming finite element function $v_h \in V_h$, there exists a unique $f_h \in L^\infty(\Omega)$ such that
	\begin{align}\label{def:HJB-rhs-f-eps}
		F_\varepsilon(f_h; x, \D^2 v_h) = 0 \text{ a.e.~in } \Omega.
	\end{align}
	The viscosity solution $u_\varepsilon$ to \eqref{def:HJB} with right-hand side 
	$f \in L^n(\Omega)$ and Dirichlet data $g \in C(\partial \Omega)$ satisfies, for any convex subset $\Omega' \Subset \Omega$, that
	\begin{align}
		\|u_\varepsilon - v_h\|_{L^\infty(\Omega)} \leq \|g - v_h\|_{L^\infty(\partial \Omega)} 
		+ \frac{c_n}{2^{1/n}n}\mathrm{diam}(\Omega')\|f - f_h\|_{L^n(\Omega')}&
		\nonumber\\
		+ \frac{c_n}{n}\mathrm{diam}(\Omega)^{(n-1)/n} \max_{x\in\overline{\Omega\setminus\Omega'}}
		\operatorname{dist}(x,\partial\Omega)^{1/n}
		\|f - f_h\|_{L^n(\Omega)} \eqqcolon \mathrm{RHS}_{\varepsilon}.&
		\label{ineq:GUB-HJB}
	\end{align}
\end{theorem}
\begin{proof}
	As in the proof of \Cref{thm:strongimpliesviscosity}, \Cref{lem:computation-stability-RHS} provides a (unique) piecewise continuous and essentially bounded function $f_h \coloneqq \xi(\D^2_{\mathrm{pw}} v_h) \in L^\infty(\Omega)$ with
	\eqref{def:HJB-rhs-f-eps}.
	\Cref{thm:strongimpliesviscosity} shows that $v_h$ is the generalized viscosity solution to \eqref{def:HJB-rhs-f-eps}. Therefore,
	the stability estimates from \Cref{lem:extended-stability} can be applied to $u_\varepsilon$ and $v_h$. First, the application of \eqref{ineq:stability-estimate-global} to the subdomain $\Omega'$ instead $\Omega$ leads to
	\begin{align*}
		\|u_\varepsilon - v_h\|_{L^\infty(\Omega')} \leq \|u_\varepsilon - v_h\|_{L^\infty(\partial \Omega')} + \frac{c_n}{2^{1/n}n}\mathrm{diam}(\Omega')\|f - f_h\|_{L^n(\Omega')}.
	\end{align*}
	Second, the local estimate \eqref{ineq:stability-estimate} with $\omega \coloneqq \Omega \setminus \Omega'$ implies
	\begin{align*}
		&\|u_\varepsilon - v_h\|_{L^\infty(\Omega \setminus \Omega')} \leq \|g - v_h\|_{L^\infty(\partial \Omega)}\\
		&\qquad +
		\frac{c_n}{n}\mathrm{diam}(\Omega)^{(n-1)/n} \max_{x\in\overline{\Omega\setminus\Omega'}}
		\operatorname{dist}(x,\partial\Omega)^{1/n}
		\|f - f_h\|_{L^n(\Omega)}.
	\end{align*}
	Since $\|u_\varepsilon - v_h\|_{L^\infty(\partial \Omega')} \leq \|u_\varepsilon - v_h\|_{L^\infty(\Omega \setminus \Omega')}$, the combination of the two previously displayed formulas concludes the proof.
\end{proof}
We point out that in both theorems of this section, it is possible to apply the stability estimate \eqref{ineq:stability-estimate} to further subsets of $\Omega$ to localize the error estimator.

\section{Numerical examples}\label{sec:num}

In this section, we apply the theory from \Cref{sec:FEM} to numerical benchmarks on the (two-dimensional) unit
square domain $\Omega \coloneqq (0,1)^2$.

\subsection{Implementation}
Some remarks on the practical realization precede the numerical benchmarks of this section.

\subsubsection{Setup}
Given $\mathcal T$ as a rectangular partition of
the domain $\Omega$ with the set $\mathcal{E}$ of edges, we choose $V_h$ to be the
Bogner--Fox--Schmit finite element space \cite{Ciarlet1978}.
It is the space of global $C^{1,1}(\overline\Omega)$ functions
that are bicubic when restricted to any element $T\in\mathcal T$.
We compute the discrete approximation in $V_h$
by approximating the regularized problem \eqref{def:F-eps}
with a Galerkin method.
In the two-dimensional setting, this yields
a strongly monotone problem with a unique discrete solution
$u_{h,\varepsilon}$ \cite{gt2021}.
Since $v_h \coloneqq u_{h,\varepsilon}$ is a $C^{1,1}(\overline{\Omega})$ function, we can apply \Cref{thm:GUB-MA} to obtain error bounds for $\|u - \Gamma_{v_h}\|_{L^\infty(\Omega)}$, which motivates an adaptive scheme as outlined below.

\subsubsection{Evaluation of the upper bound of Theorem~\ref{thm:GUB-MA}}
We proceed as follows for the computation of the right-hand side $\mathrm{RHS}_0$ of \eqref{ineq:GUB-MA}.\medskip

\emph{Integration of $f - f_h$ for $f_h \coloneqq 2 \chi_{\mathcal{C}_{v_h}} (\det \D^2_\mathrm{pw} v_h)^{1/2}$.}
The integral $\|f - f_h\|_{L^2(\omega)}$ for any subset $\omega \subset \Omega$ is computed via numerical integration.
Given a set of Gauss points $\mathcal{N}_\ell$ associated to the degree of exact integration $\ell$, this reads
\begin{align}\label{eq:numerical-integration-RHS}
	\sum_{T \in \mathcal{T}} \sum_{x \in \mathcal{N}_\ell \cap T \cap \omega} \operatorname{meas}(T) w_{\ell,T}(x)(f(x) - 2\chi_{\mathcal{C}_{v_h}}(x) (\det \D^2_\mathrm{pw} v_h(x))^{1/2})^2
\end{align}
with some positive weight function $w_{\ell,T} \in L^\infty(T)$.
A point $x \in \mathcal{N}_\ell$ is in the contact set $\mathcal{C}_{v_h}$ of $v_h$ if (and only if)
\begin{align}\label{ineq:contact-set-condition}
	0 \leq v_h(z) - v_h(x) - \nabla v_h(x) \cdot (z - x) \quad\text{for all } z \in \Omega
\end{align}
(because $\partial \Gamma_{v_h} (x) = \{\nabla v_h(x)\}$ for any $x \in \Omega \cap \mathcal{C}_{v_h}$ from the proof of \Cref{thm:GUB-MA}).
While this condition can be checked explicitly, it leads to a global problem for each Gauss point, which may become rather expensive.
Instead, \eqref{ineq:contact-set-condition} is verified at only a finite number of points, e.g., $z \in \mathcal{V}_\ell \coloneqq \mathcal{N}_\ell \cup \mathcal{N}_\ell^b$, where $\mathcal{N}_\ell^b \subset \partial \Omega$ is a discrete subset of $\partial \Omega$.
The set of points $\mathcal{V}_\ell$ create a quasi-uniform refinement $\mathcal{T}_\ell$ of the partition $\mathcal{T}$ into triangles and we assume that the mesh-size of $\mathcal{T}_\ell$ tends to zero as $\ell \to \infty$.
Let $\mathrm{I}_\ell v_h$ denote the nodal interpolation of $v_h$ w.r.t.~the mesh $\mathcal{T}_\ell$.
We replace the function $\chi_{\mathcal{C}_{v_h}}$ in \eqref{eq:numerical-integration-RHS} by
the indicator function $\chi_{\mathcal{C}^\ell_{v_h}}$ of the set
\begin{align*}
	\mathcal{C}^\ell_{v_h} \coloneqq \mathcal{C}_{\mathrm{I}_\ell v_h} \cap \{x \in \Omega \setminus \cup \mathcal{E}: \D^2_\mathrm{pw} v_h(x) \geq 0 \text{ is positive semi-definite}\}.
\end{align*}
In practice, the numerical integration formula for $\|f - f_h\|_{L^2(\omega)}$ reads
\begin{align}\label{eq:numerical-integration-appr-RHS}
	\sum_{T \in \mathcal{T}} \sum_{x \in \mathcal{N}_\ell \cap T \cap \omega} \operatorname{meas}(T) w_{\ell,T}(x)(f(x) - 2\chi_{\mathcal{C}^\ell_{v_h}}(x) (\det \D^2_\mathrm{pw} v_h(x))^{1/2})^2.
\end{align}
The convex envelope $\Gamma_{\mathrm{I}_\ell v_h}$ of $\mathrm{I}_\ell v_h$ can be computed, for instance, by the quickhull algorithm \cite{BarberDobkinHuhdanpaa1996}.
Therefore, it is straight-forward to compute \eqref{eq:numerical-integration-appr-RHS}.
We note that if $x \in \mathcal{C}_{v_h} \cap \mathcal{N}_\ell$, then \eqref{ineq:contact-set-condition} holds for any $z \in \mathcal{V}_\ell$.
Since the convex envelope of the continuous piecewise affine function $\mathrm{I}_\ell v_h$ only depends on the nodal values of $v_h$, this implies $x \in \mathcal{C}^\ell_{v_h} \cap \mathcal{N}_\ell$.
However, the reverse is not true. 
Hence, \eqref{eq:numerical-integration-appr-RHS} and \eqref{eq:numerical-integration-RHS} may not coincide.
From the uniform convergence of $\mathrm{I}_\ell v_h$ to $v_h$ as $\ell \to \infty$, we deduce
\begin{align*}
	\limsup_{\ell \to \infty} \mathcal{C}^\ell_{v_h} \coloneqq \cap_{\ell \in \mathbb{N}} \cup_{k \geq \ell} \mathcal{C}^\ell_{v_h} \subset \mathcal{C}_{v_h},
\end{align*}
cf.~\cite[Lemma A.1]{CaffarelliCrandallKocanSwiech1996}.
Given any $\delta > 0$, this implies $\mathcal{C}^\ell_{v_h} \setminus \mathcal{C}_{v_h} \subset \{x \in \Omega\setminus \mathcal{C}_{v_h} : \mathrm{dist}(x,\mathcal{C}_{v_h}) \leq \delta\}$ for sufficiently large $\ell$.
Therefore, the set of all points $x \in \mathcal{N}_\ell$ with $\chi_{\mathcal{C}_{v_h}} \neq \chi_{\mathcal{C}^\ell_{v_h}}(x)$ is a subset of $\mathcal{C}^\ell_{v_h} \setminus \mathcal{C}_{v_h}$, whose Lebesgue measure vanishes in the limit as $\ell \to \infty$.
In conclusion, the limits of \eqref{eq:numerical-integration-RHS} and \eqref{eq:numerical-integration-appr-RHS} coincide.\medskip

\emph{Computation of $\mu \coloneqq \limsup_{x \to \partial \Omega} |(g - \Gamma_{v_h})(x)|$.}
The boundary residual $\mu$ is approximated by $\|g - \Gamma_{\mathrm{I}_\ell v_h}\|_{L^\infty(\partial \Omega)}$.
Since $\Gamma_{v_h} \leq \mathrm{I}_\ell v_h$ and $\mathrm{I}_\ell v_h$ is piecewise affine, $\Gamma_{v_h} \leq \Gamma_{\mathrm{I}_\ell v_h}$ holds in $\Omega$.
On the other hand, we have $\lim_{\ell \to \infty} \|v_h - \mathrm{I}_\ell v_h\|_{L^\infty(\Omega)} = 0$.
Hence, any supporting hyperplane $a_{x}$ of $\Gamma_{\mathrm{I}_\ell v_h}$ at $x \in \Omega$ satisfies $a_x - \delta_\ell \leq v_h$ in $\Omega$ with $\delta_\ell \coloneqq \|v_h - \mathrm{I}_\ell v_h\|_{L^\infty(\Omega)}$.
Since $a_x - \delta_\ell$ is an affine function, $\Gamma_{\mathrm{I}_\ell v_h}(x) - \delta_\ell = a_x(x) - \delta_\ell \leq \Gamma_{v_h}(x)$.
We conclude $\Gamma_{\mathrm{I}_\ell v_h} - \delta_\ell \leq \Gamma_{v_h} \leq \Gamma_{\mathrm{I}_\ell v_h}$ in $\Omega$.
In particular, $\lim_{\ell \to \infty} \|g - \Gamma_{\mathrm{I}_\ell v_h}\|_{L^\infty(\partial \Omega)} = \mu$.\medskip

\emph{Choice of $\Omega'$.} Let $\delta \coloneqq \min_{E \in \mathcal{E}} h_E$ denote the minimal edge length of the mesh $\mathcal{T}$.
For all integers $0 \leq j < 1/(2\delta)$, define $\Omega_{j\delta} \coloneqq \{x \in \Omega: \mathrm{dist}(x,\partial \Omega) \geq j\delta\}$.
It seems canonical to choose $\Omega' \coloneqq \Omega_{j\delta}$, where $j$ is the index that minimizes $\mathrm{RHS}_0$.
However, this choice may lead to significant computational effort.
From the interior regularity of Alexandrov solutions \cite{Caffarelli1990},
we can expect that the error is concentrated on the boundary and so, the best $j$ will be close to one.
Accordingly, the smallest $j \geq 0$ is chosen so that $\mathrm{RHS}_0$ with $\Omega' \coloneqq \Omega_{(j+1)\delta}$ is larger than $\mathrm{RHS}_0$ with $\Omega' \coloneqq \Omega_{j\delta}$.

\subsubsection{Adaptive marking strategy}
We define the refinement indicator
\begin{align*}
	\eta(T) \coloneqq j\delta\sqrt{2}\|f-f_h\|_{L^2(T)}^2 + (1-2j\delta)^2\|f-f_h\|_{L^2(T \cap \Omega_{j\delta})}^2
\end{align*}
for any $T \in \mathcal{T}$, where the scaling in $\delta$ arises from \eqref{ineq:GUB-MA} with $n = 2$. Let $\sigma \coloneqq \mathrm{RHS}_0 - \mu$ denote the remaining contributions of $\mathrm{RHS}_0$, where $\mu = \limsup_{x \to \partial \Omega} |(g - \Gamma_{u_{h,\varepsilon}})(x)|$ from above.
If $\sigma/10 < \|g - u_{h,\varepsilon}\|_{L^\infty(\partial \Omega)}$, then we mark one fifth of all boundary edges $E \in \mathcal{E}$ with the largest contributions $\|g - u_{h,\varepsilon}\|_{L^\infty(E)}$.
Otherwise, we mark a set $\mathcal{M}$ of rectangles with minimal cardinality so that
\begin{align*}
	\frac{1}{2}\sum_{T \in \mathcal{T}} \eta(T) \leq \sum_{T \in \mathcal{M}} \eta(T).
\end{align*}

\subsubsection{Displayed quantities}
The convergence history plots display the errors $\|u - u_{h,\varepsilon}\|_{L^\infty(\Omega)}$, $\mathrm{LHS} \coloneqq \|u - \Gamma_{u_{h,\varepsilon}}\|_{L^\infty(\Omega)}$ as well as the error estimator $\mathrm{RHS}_0$ against the number of degrees of freedom $\mathrm{ndof}$ in a log-log plot. (We note that $\mathrm{ndof}$ scales 
like $h_\mathrm{\max}^{-2}$ on uniformly refined meshes.)
Whenever the solution $u$ is sufficiently smooth, the errors $\|u - u_{h,\varepsilon}\|_{H^1(\Omega)}$ and $\|u - u_{h,\varepsilon}\|_{H^2(\Omega)}$ are also displayed.
Solid lines in the convergence history plots indicate adaptive mesh-refinements, while dashed lines are associated with uniform mesh-refinements.
The experiments are carried out for the regularization parameters $\varepsilon = 10^{-3}$ in the first two experiments and $\varepsilon = 10^{-4}$ for the third experiment.
For a numerical comparison of various $\varepsilon$, we refer to \cite{gt2021}.

\subsection{Regular solution}
In this example from \cite{DeanGlowinski2006}, the exact solution $u$ is given by
$$
u(x) = \frac{(2|x|)^{3/2}}{3}
$$
with $f(x) = 1/|x|$.
The solution belongs to $H^{5/2-\nu}(\Omega)$
for any $\nu>0$, but not to $C^2(\overline{\Omega})$.
It is proven in \cite{gt2021} that $u$ is the viscosity solution to $F_\varepsilon(f;x,\D^2 u) = 0$ in $\Omega$ for any regularization parameter $0 < \varepsilon \leq 1/3$.
Accordingly, we observed no visual differences in the convergence history plots for different $0 < \varepsilon \leq 1/3$.
\Cref{f:exp1-conv} displays the convergence rates $0.8$ for $\|u - u_{h,\varepsilon}\|_{L^\infty(\Omega)}$ and $\mathrm{RHS}$, $3/4$ for $\|u - u_{h,\varepsilon}\|_{H^1(\Omega)}$, and $1/4$ for $\|u-u_{h,\varepsilon}\|_{H^2(\Omega)}$ on uniform meshes.
The adaptive algorithm refines towards the singularity of $u$ at $0$ and leads to improved convergence rates for all displayed quantities.
We observe the rate $1.75$ for $\|u - u_{h,\varepsilon}\|_{L^\infty(\Omega)}$, $1$ for $\mathrm{LHS}$, $\mathrm{RHS}_0$, and $\|u - u_{h,\varepsilon}\|_{H^2(\Omega)}$, and $1.5$ for $\|u-u_{h,\varepsilon}\|_{H^1(\Omega)}$.
It is also worth noting that $\mathrm{RHS}_0$ seems to be efficient on adaptive meshes.

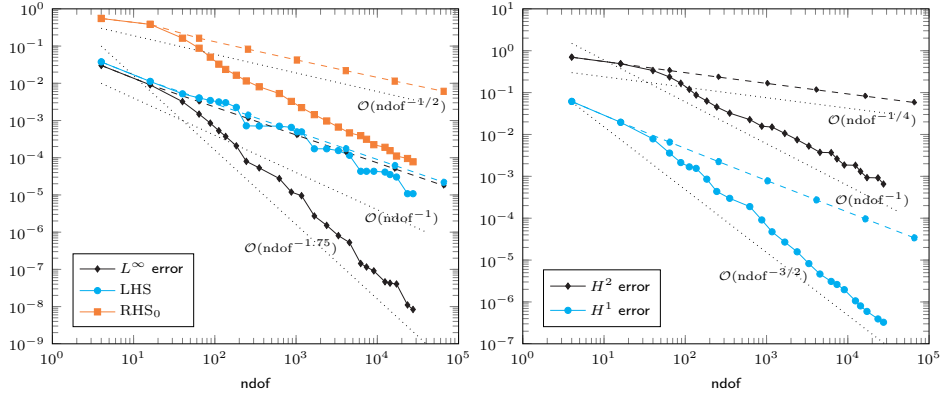
\begin{figure}
	\centering
	\begin{subfigure}[b]{0.48\textwidth}
		\begin{tikzpicture}[scale=0.78]
			\begin{loglogaxis}[legend pos=south west,legend cell align=left,
				legend style={fill=none},
				ymin=1e-9,ymax=1e0,
				xmin=1e0,xmax=1e5,
				ytick={1e-9,1e-8,1e-7,1e-6,1e-5,1e-4,1e-3,1e-2,1e-1,1e0},
				xtick={1e0,1e1,1e2,1e3,1e4,1e5}]
				\pgfplotsset{
					cycle list={%
						{Black, mark=diamond*, mark size=1.5pt},
						{Cyan, mark=*, mark size=1.5pt},
						{Orange, mark=square*, mark size=1.5pt},
						{Black, dashed, mark=diamond*, mark size=1.5pt},
						{Cyan, dashed, mark=*, mark size=1.5pt},
						{Orange, dashed, mark=square*, mark size=1.5pt},
					},
					legend style={
						at={(0.05,0.05)}, anchor=south west},
					font=\sffamily\scriptsize,
					xlabel=ndof,ylabel=
				}
				\addplot+ table[x=ndof,y=Linferr]{convhist_ex1_adap_epsi0.001.dat};
				\addlegendentry{$L^\infty$ error}
				\addplot+ table[x=ndof,y=LHS]{convhist_ex1_adap_epsi0.001.dat};
				\addlegendentry{$\mathrm{LHS}$}
				\addplot+ table[x=ndof,y=eta2]{convhist_ex1_adap_epsi0.001.dat};
				\addlegendentry{$\mathrm{RHS}_0$}
				\addplot+ table[x=ndof,y=Linferr]{convhist_ex1_unif_epsi0.001.dat};
				\addplot+ table[x=ndof,y=LHS]{convhist_ex1_unif_epsi0.001.dat};
				\addplot+ table[x=ndof,y=eta2]{convhist_ex1_unif_epsi0.001.dat};
				\addplot+ [dotted,mark=none,black] coordinates{(4,1e-2) (4e4,1e-6)};
				\node(z) at  (axis cs:2e4,1e-6)
				[above] {\tiny $\mathcal O (\mathrm{ndof}^{-1})$};
				\addplot+ [dotted,mark=none,black] coordinates{(4,1e-1) (4e4,1e-9)};
				\node(z) at  (axis cs:4e3,1e-6)
				[below left] {\tiny $\mathcal O (\mathrm{ndof}^{-1.75})$};
				\addplot+ [dotted,mark=none,black] coordinates{(4,3e-1) (4e4,3e-3)};
				\node(z) at  (axis cs:2e4,1e-3)
				[above] {\tiny $\mathcal O (\mathrm{ndof}^{-1/2})$};
			\end{loglogaxis}
		\end{tikzpicture}
	\end{subfigure}
	\begin{subfigure}[b]{0.48\textwidth}
		\begin{tikzpicture}[scale=0.78]
			\begin{loglogaxis}[legend pos=south west,legend cell align=left,
				legend style={fill=none},
				ymin=1e-7,ymax=1e1,
				xmin=1e0,xmax=1e5,
				ytick={1e-7,1e-6,1e-5,1e-4,1e-3,1e-2,1e-1,1e0,1e1},
				xtick={1e0,1e1,1e2,1e3,1e4,1e5}]
				\pgfplotsset{
					cycle list={%
						{Black, mark=diamond*, mark size=1.5pt},
						{Cyan, mark=*, mark size=1.5pt},
						{Black, dashed, mark=diamond*, mark size=1.5pt},
						{Cyan, dashed, mark=*, mark size=1.5pt},
					},
					legend style={
						at={(0.05,0.05)}, anchor=south west},
					font=\sffamily\scriptsize,
					xlabel=ndof,ylabel=
				}
				\addplot+ table[x=ndof,y=H2error]{convhist_ex1_adap_epsi0.001.dat};
				\addlegendentry{$H^2$ error}
				\addplot+ table[x=ndof,y=H1error]{convhist_ex1_adap_epsi0.001.dat};
				\addlegendentry{$H^1$ error}
				\addplot+ table[x=ndof,y=H2error]{convhist_ex1_unif_epsi0.001.dat};
				\addplot+ table[x=ndof,y=H1error]{convhist_ex1_unif_epsi0.001.dat};
				\addplot+ [dotted,mark=none,black] coordinates{(4,1.5e-0) (4e4,1.5e-4)};
				\node(z) at  (axis cs:2e4,1.2e-4)
				[above] {\tiny $\mathcal O (\mathrm{ndof}^{-1})$};
				\addplot+ [dotted,mark=none,black] coordinates{(4,6e-2) (4e4,6e-8)};
				\node(z) at  (axis cs:4e3,1e-5)
				[below left] {\tiny $\mathcal O (\mathrm{ndof}^{-3/2})$};
				\addplot+ [dotted,mark=none,black] coordinates{(4,3e-1) (4e4,3e-2)};
				\node(z) at  (axis cs:2e4,1e-2)
				[above] {\tiny $\mathcal O (\mathrm{ndof}^{-1/4})$};
			\end{loglogaxis}
		\end{tikzpicture}
	\end{subfigure}
	\caption{Convergence history for the first experiment with $\varepsilon = 10^{-3}$.
		\label{f:exp1-conv}
	}
\end{figure}

\subsection{Convex envelope of boundary data}
In the second example, we approximate the exact solution
\begin{align*}
	u(x,y) \coloneqq |x - 1/2|
\end{align*}
to $\det \D^2 u = 0$ in $\Omega$, which is the largest convex function with prescribed boundary data.
The solution belongs to $H^{3/2-\delta}(\Omega)$ for any $\delta$ > 0, but not to $H^{3/2}(\Omega)$.
It was observed in \cite{gt2021} that the regularization error of $u-u_\varepsilon$ dominates the discretization error $u - u_{h,\varepsilon}$ on finer meshes.
Therefore, the errors $\|u - u_{h,\varepsilon}\|_{L^\infty(\Omega)}$ and $\|u - u_{h,\varepsilon}\|_{H^1(\Omega)}$ stagnate at a certain value (depending on $\varepsilon$) as displayed in \Cref{f:exp2-conv}.
However, $\mathrm{LHS}$ converges with convergence rate $1/2$ on uniform meshes even for fixed $\varepsilon$.
At first glance on the discrete solution shown in \Cref{f:exp2-discrete-solution},
we can expect that the maximum of $|u - u_{h,\varepsilon}|$ is attained along the line $\mathrm{conv}\{(1/2,0),(1/2,1)\}$.
This error depends on the regularization parameter and only vanishes in the limit as $\varepsilon \to 0$, but the convex envelope of $u_{h,\varepsilon}$ provides an accurate approximation of $u$ along this line.
In fact, \Cref{f:exp2-mesh} shows that the adaptive algorithm refines towards the points $(1/2,0)$ and $(1/2,1)$, but the whole line $\mathrm{conv}\{(1/2,0),(1/2,1)\}$ is only of minor interest.
We observe the improved convergence rate $2.5$ for $\mathrm{LHS}$ on adaptive meshes.
The guaranteed upper bound $\mathrm{RHS}_0$ can provide an accurate estimate of $\mathrm{LHS}$, but seems to oscillate due to the nature of the problem.
The goal of the adaptive algorithm is the reduction of $\mathrm{RHS}_0$, which consists of the error $\|f - f_h\|_{L^2(\Omega)}$ in the Monge--Amp\`ere measures and of some boundary data approximation error.
Thanks to the additional regularization provided by the convex envelope, $\|f - f_h\|_{L^2(\Omega)}$ is concentrated at the points $(1/2,0)$ and $(1/2,1)$, but becomes very small after some mesh-refining steps.
We even observed in \Cref{f:exp2-conv} that $\mathrm{LHS} = \mathrm{RHS}_0$ on two meshes, i.e., $\|f - f_h\|_{L^2(\Omega)} = 0$.
Then $\mathrm{RHS}_0$ is dominated by the data boundary approximation error and leads to mesh refinements on the boundary.
This may result in significant changes in the Monge--Amp\`ere measure of $\Gamma_{u_{h,\varepsilon}}$, because the convex envelope of the discrete function $u_{h,\varepsilon}$ depends heavily on its values on the boundary in this class of problems.

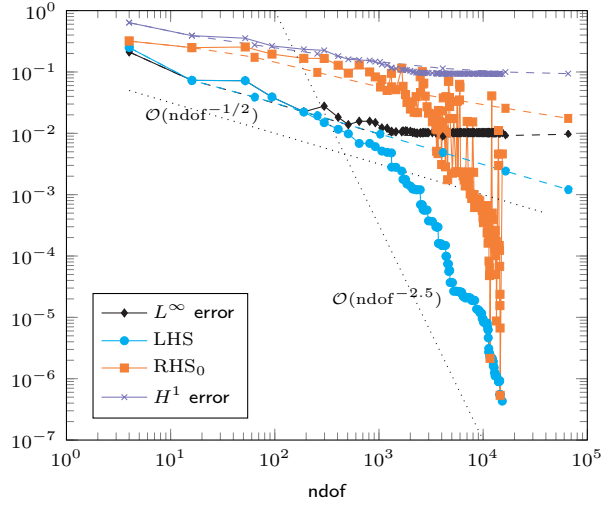
\begin{figure}
	\centering
	\begin{tikzpicture}
		\begin{loglogaxis}[legend pos=south west,legend cell align=left,
			legend style={fill=none},
			ymin=1e-7,ymax=1e0,
			xmin=1e0,xmax=1e5,
			ytick={1e-7,1e-6,1e-5,1e-4,1e-3,1e-2,1e-1,1e0},
			xtick={1e0,1e1,1e2,1e3,1e4,1e5}]
			\pgfplotsset{
				cycle list={%
					{Black, mark=diamond*, mark size=1.5pt},
					{Cyan, mark=*, mark size=1.5pt},
					{Orange, mark=square*, mark size=1.5pt},
					{Periwinkle, mark=x, mark size=1.5pt},
					{Black, dashed, mark=diamond*, mark size=1.5pt},
					{Cyan, dashed, mark=*, mark size=1.5pt, fill opacity=10},
					{Orange, dashed, mark=square*, mark size=1.5pt},
					{Periwinkle, dashed, mark=x, mark size=1.5pt},
				},
				legend style={
					at={(0.05,0.05)}, anchor=south west},
				font=\sffamily\scriptsize,
				xlabel=ndof,ylabel=
			}
			\addplot+ table[x=ndof,y=Linferr]{convhist_ex2_adap_epsi0.001.dat};
			\addlegendentry{$L^\infty$ error}
			\addplot+ table[x=ndof,y=LHS]{convhist_ex2_adap_epsi0.001.dat};
			\addlegendentry{$\mathrm{LHS}$}
			\addplot+ table[x=ndof,y=eta2]{convhist_ex2_adap_epsi0.001.dat};
			\addlegendentry{$\mathrm{RHS}_0$}
			\addplot+ table[x=ndof,y=H1error]{convhist_ex2_adap_epsi0.001.dat};
			\addlegendentry{$H^1$ error}
			\addplot+ table[x=ndof,y=Linferr]{convhist_ex2_unif_epsi0.001.dat};
			\addplot+ table[x=ndof,y=LHS]{convhist_ex2_unif_epsi0.001.dat};
			\addplot+ table[x=ndof,y=eta2]{convhist_ex2_unif_epsi0.001.dat};
			\addplot+ table[x=ndof,y=H1error]{convhist_ex2_unif_epsi0.001.dat};
			\addplot+ [dotted,mark=none] coordinates{(4,5e-2) (4e4,5e-4)};
			\node(z) at  (axis cs:2e1,1e-2)
			[above] {\tiny $\mathcal O (\mathrm{ndof}^{-1/2})$};
			\addplot+ [dotted,mark=none] coordinates{(1e2,1e-0) (1e4,1e-7)};
			\node(z) at  (axis cs:1.2e3,1e-5)
			[above] {\tiny $\mathcal O (\mathrm{ndof}^{-2.5})$};
		\end{loglogaxis}
	\end{tikzpicture}
	\caption{Convergence history for the second experiment with $\varepsilon = 10^{-4}$.}
	\label{f:exp2-conv}
\end{figure}
\begin{figure}
	\centering
	\includegraphics[scale=0.18]{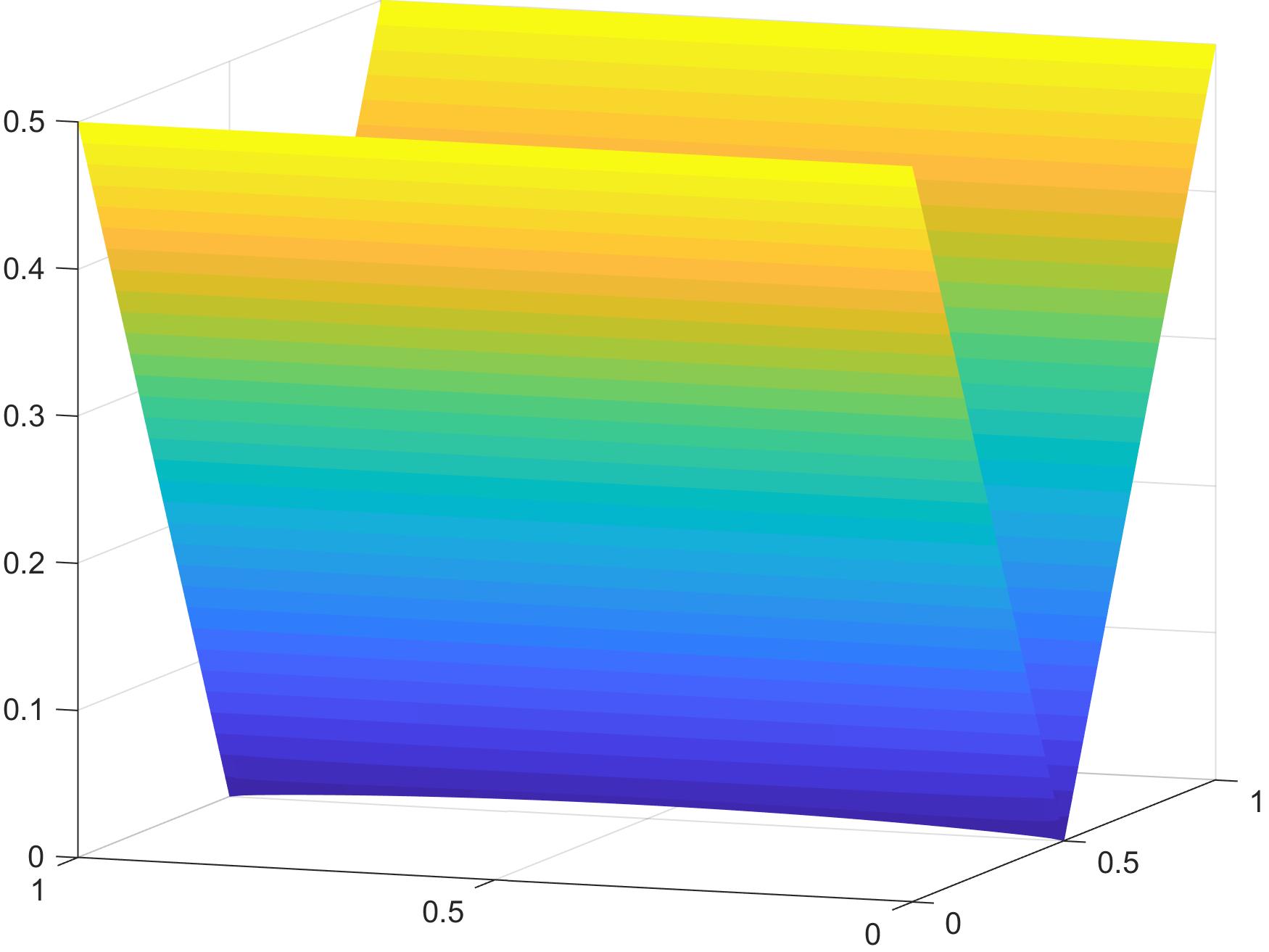}
	\caption{Discrete solution on a uniform mesh with 4225 nodes.}
	\label{f:exp2-discrete-solution}
\end{figure}
\begin{figure}
	\centering
	\includegraphics[scale=0.25]{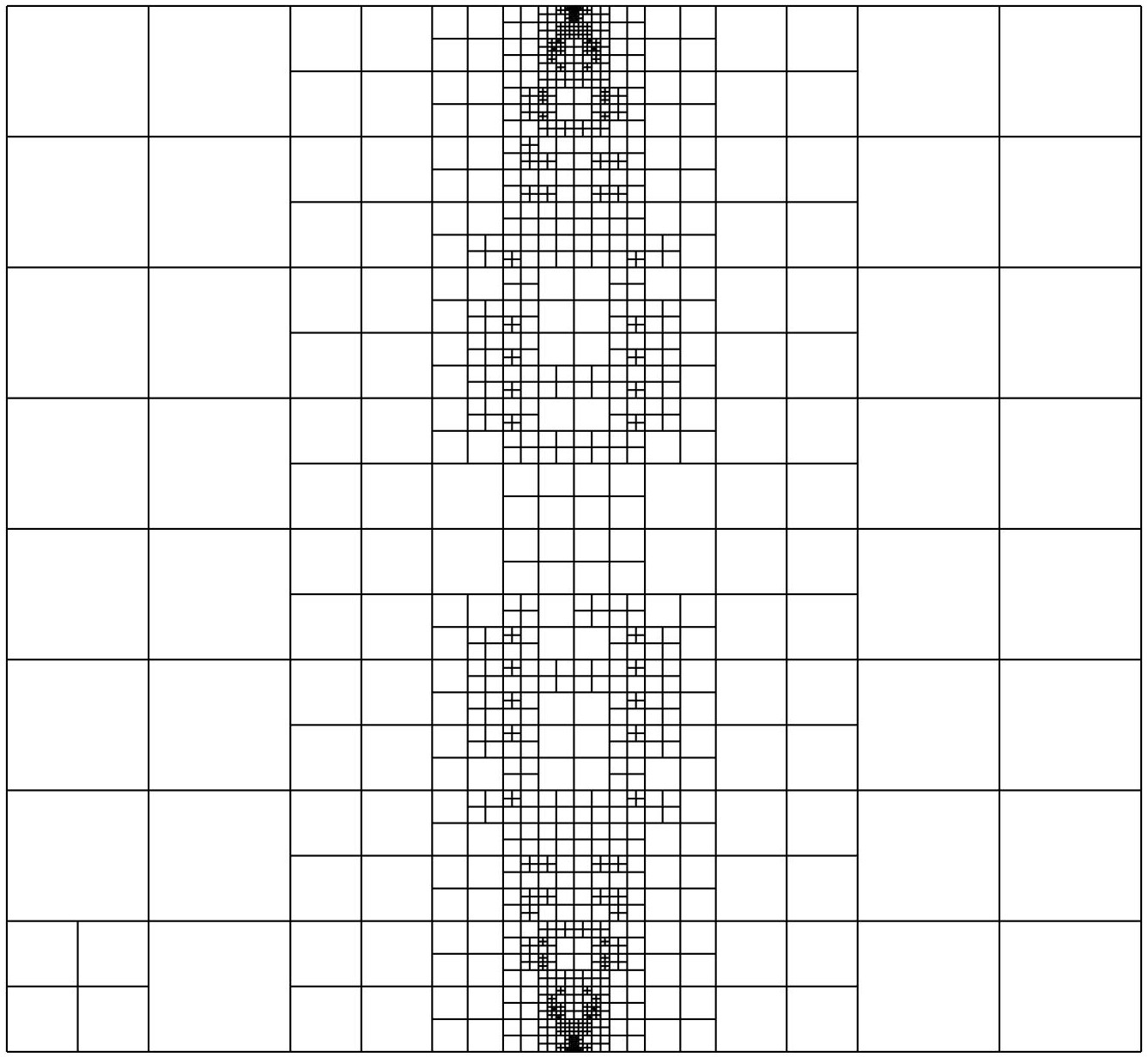}
	\caption{Adaptive mesh with 1907 nodes for the second experiment.}
	\label{f:exp2-mesh}
\end{figure}

\subsection{Nonsmooth exact solution}
In this example, the function
\begin{align*}
	u(x,y) \coloneqq -\big(\sin(\pi x)^{-1} + \sin(\pi y)^{-1}\big)^{-1}
\end{align*}
is the solution to the Monge--Amp\`ere equation \eqref{def:Monge-Ampere} with homogenous boundary data and right-hand side
\begin{align*}
	f(x,y) = \frac{4\pi^2\sin(\pi x)^2\sin(\pi y)^2(2-\sin(\pi x)\sin(\pi y))}{(\sin(\pi x) + \sin(\pi y))^4}.
\end{align*}
The function $u$ belongs to $C^2(\Omega) \cap H^{2-\delta}(\Omega)$ for all $\delta > 0$, but neither to $H^{2}(\Omega)$ nor $C^2(\overline{\Omega})$.
The convergence history is displayed in \Cref{f:exp3-conv}.
Notice from \Cref{prop:continuity-convex-envelope} that $\mathrm{RHS}_0$ consists solely of the error in the Monge--Amp\`ere measures.
In this example, $f$ exhibits strong oscillations at the four corners of the domain $\Omega$ and the adaptive algorithm seems to solely refine towards these corners as displayed in \Cref{f:exp3-mesh}.
While $\mathrm{RHS}_0$ converges on uniform meshes (although with a slow rate), there is only a marginal reduction of $\mathrm{RHS}_0$ for adaptive computation.
We can conclude that the discrete approximation cannot resolve the infinitesimal oscillation of the Monge--Amp\`ere measure of $u$ properly.
This results in the stagnation of $\|u - u_{h,\varepsilon}\|_{L^\infty(\Omega)}$ and $\mathrm{LHS}$ at an early level in comparison to uniform mesh refinements.
However, we also observed that the stagnation point depends on the maximal mesh-size.
In fact, if we start from an initial uniform mesh with a small mesh-size $h_0$, significant improvements of $\mathrm{RHS}_0$ are obtained on the first levels of adaptive mesh refinements as displayed in \Cref{f:exp3-conv-inital-mesh}.
Undisplayed experiments show the same behaviour for $\|u - u_{h,\varepsilon}\|_{L^\infty(\Omega)}$.
This leads us to believe that, in this example, a combination of uniform and adaptive mesh-refining strategy provides the best results.

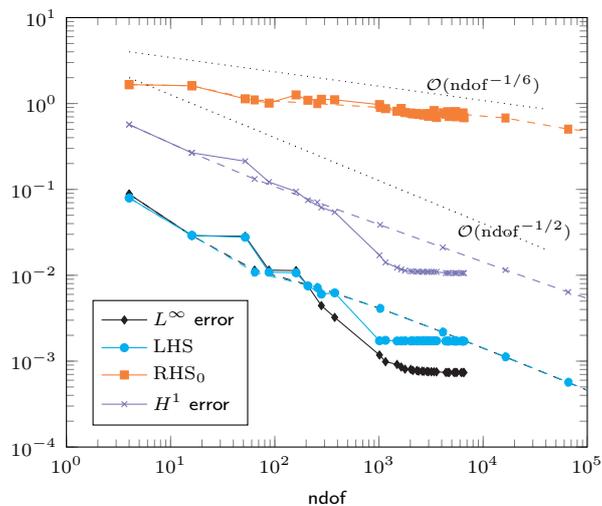
\begin{figure}
	\centering
	\begin{tikzpicture}
		\begin{loglogaxis}[legend pos=south west,legend cell align=left,
			legend style={fill=none},
			ymin=1e-4,ymax=1e1,
			xmin=1e0,xmax=1e5,
			ytick={1e-4,1e-3,1e-2,1e-1,1e0,1e1},
			xtick={1e0,1e1,1e2,1e3,1e4,1e5}]
			\pgfplotsset{
				cycle list={%
					{Black, mark=diamond*, mark size=1.5pt},
					{Cyan, mark=*, mark size=1.5pt},
					{Orange, mark=square*, mark size=1.5pt},
					{Periwinkle, mark=x, mark size=1.5pt},
					{Black, dashed, mark=diamond*, mark size=1.5pt},
					{Cyan, dashed, mark=*, mark size=1.5pt, fill opacity=10},
					{Orange, dashed, mark=square*, mark size=1.5pt},
					{Periwinkle, dashed, mark=x, mark size=1.5pt},
				},
				legend style={
					at={(0.05,0.05)}, anchor=south west},
				font=\sffamily\scriptsize,
				xlabel=ndof,ylabel=
			}
			\addplot+ table[x=ndof,y=Linferr]{convhist_ex3_adap_epsi0.0001.dat};
			\addlegendentry{$L^\infty$ error}
			\addplot+ table[x=ndof,y=LHS]{convhist_ex3_adap_epsi0.0001.dat};
			\addlegendentry{$\mathrm{LHS}$}
			\addplot+ table[x=ndof,y=eta2]{convhist_ex3_adap_epsi0.0001.dat};
			\addlegendentry{$\mathrm{RHS}_0$}
			\addplot+ table[x=ndof,y=H1error]{convhist_ex3_adap_epsi0.0001.dat};
			\addlegendentry{$H^1$ error}
			\addplot+ table[x=ndof,y=Linferr]{convhist_ex3_unif_epsi0.0001.dat};
			\addplot+ table[x=ndof,y=LHS]{convhist_ex3_unif_epsi0.0001.dat};
			\addplot+ table[x=ndof,y=eta2]{convhist_ex3_unif_epsi0.0001.dat};
			\addplot+ table[x=ndof,y=H1error]{convhist_ex3_unif_epsi0.0001.dat};
			\addplot+ [dotted,mark=none] coordinates{(4,2e-0) (4e4,2e-2)};
			\node(z) at  (axis cs:2e4,2e-2)
			[above] {\tiny $\mathcal O (\mathrm{ndof}^{-1/2})$};
			\addplot+ [dotted,mark=none] coordinates{(4,4e-0) (4e4,8.6e-1)};
			\node(z) at  (axis cs:1e4,1e-0)
			[above] {\tiny $\mathcal O (\mathrm{ndof}^{-1/6})$};
		\end{loglogaxis}
	\end{tikzpicture}
	\caption{Convergence history for the third experiment with $\varepsilon = 10^{-4}$.}
	\label{f:exp3-conv}
\end{figure}
\begin{figure}
	\centering
	\includegraphics[scale=0.25]{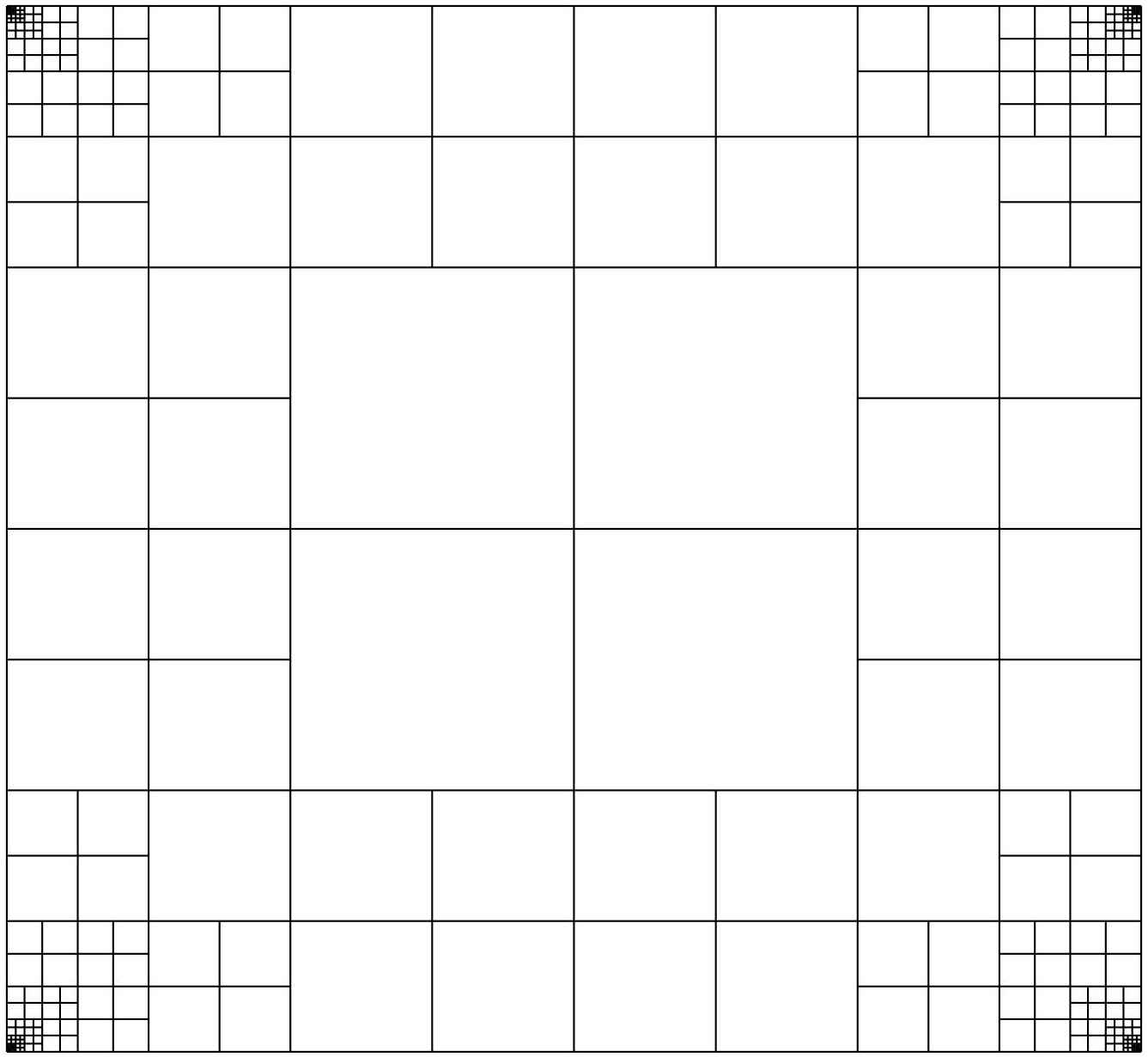}
	\caption{Adaptive mesh with 1351 nodes for the third experiment.}
	\label{f:exp3-mesh}
\end{figure}
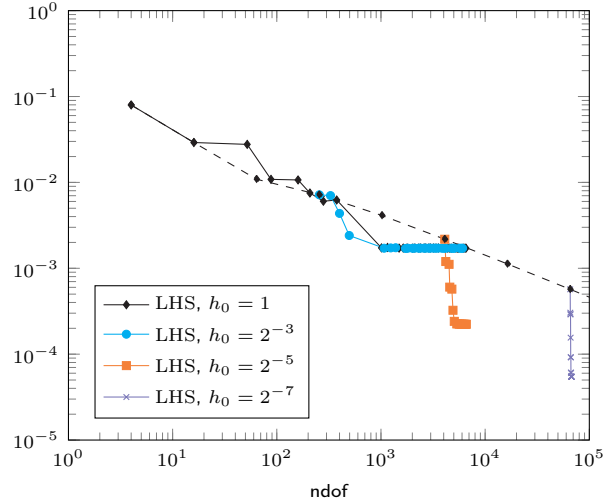
\begin{figure}
	\centering
	\begin{tikzpicture}
		\begin{loglogaxis}[legend pos=south west,legend cell align=left,
			legend style={fill=none},
			ymin=1e-5,ymax=1e0,
			xmin=1e0,xmax=1e5,
			ytick={1e-5,1e-4,1e-3,1e-2,1e-1,1e0,1e1},
			xtick={1e0,1e1,1e2,1e3,1e4,1e5}]
			\pgfplotsset{
				cycle list={%
					{Black, mark=diamond*, mark size=1.5pt},
					{Cyan, mark=*, mark size=1.5pt},
					{Orange, mark=square*, mark size=1.5pt},
					{Periwinkle, mark=x, mark size=1.5pt},
					{Black, dashed, mark=diamond*, mark size=1.5pt},
					{Cyan, dashed, mark=*, mark size=1.5pt, fill opacity=10},
					{Orange, dashed, mark=square*, mark size=1.5pt},
					{Periwinkle, dashed, mark=x, mark size=1.5pt},
				},
				legend style={
					at={(0.05,0.05)}, anchor=south west},
				font=\sffamily\scriptsize,
				xlabel=ndof,ylabel=
			}
			\addplot+ table[x=ndof,y=LHS]{convhist_ex3_adap_epsi0.0001.dat};
			\addlegendentry{$\mathrm{LHS}$, $h_0 = 1$}
			\addplot+ table[x=ndof,y=LHS]{convhist_ex3_adap_epsi0.0001_h3.dat};
			\addlegendentry{$\mathrm{LHS}$, $h_0 = 2^{-3}$}
			\addplot+ table[x=ndof,y=LHS]{convhist_ex3_adap_epsi0.0001_h5.dat};
			\addlegendentry{$\mathrm{LHS}$, $h_0 = 2^{-5}$}
			\addplot+ table[x=ndof,y=LHS]{convhist_ex3_adap_epsi0.0001_h7.dat};
			\addlegendentry{$\mathrm{LHS}$, $h_0 = 2^{-7}$}
			\addplot+ table[x=ndof,y=LHS]{convhist_ex3_unif_epsi0.0001.dat};
		\end{loglogaxis}
	\end{tikzpicture}
	\caption{Convergence history of $\mathrm{LHS}$ for the third experiment with $\varepsilon = 10^{-4}$ and different initial meshes.}
	\label{f:exp3-conv-inital-mesh}
\end{figure}

\bibliography{references}

\providecommand{\bysame}{\leavevmode\hbox to3em{\hrulefill}\thinspace}
\providecommand{\MR}{\relax\ifhmode\unskip\space\fi MR }
\providecommand{\MRhref}[2]{%
  \href{http://www.ams.org/mathscinet-getitem?mr=#1}{#2}
}
\providecommand{\href}[2]{#2}
\begin{thebibliography}{10}

\bibitem{Alexandrov1958}
A.~D. Alexandrov, \emph{Dirichlet’s problem for the equation $\det \|z_{ij}\|
  = \varphi(z_1,\dots,z_n,z,x_1,\dots,x_n$ i}, Vestnik Leningrad. Univ. Ser.
  Mat. Meh. Astr. \textbf{13} (1958), no.~1, 1--25.

\bibitem{BarberDobkinHuhdanpaa1996}
C.~Bradford Barber, David~P. Dobkin, and Hannu Huhdanpaa, \emph{The quickhull
  algorithm for convex hulls}, ACM Trans. Math. Software \textbf{22} (1996),
  no.~4, 469--483. \MR{1428265}

\bibitem{CaffarelliCrandallKocanSwiech1996}
L.~Caffarelli, M.~G. Crandall, M.~Kocan, and A.~\'Swi\k{e}ch, \emph{On
  viscosity solutions of fully nonlinear equations with measurable
  ingredients}, Comm. Pure Appl. Math. \textbf{49} (1996), no.~4, 365--397.
  \MR{1376656}

\bibitem{Caffarelli1990}
Luis~A. Caffarelli, \emph{Interior {$W^{2,p}$} estimates for solutions of the
  {M}onge-{A}mp\`ere equation}, Ann. of Math. (2) \textbf{131} (1990), no.~1,
  135--150. \MR{1038360}

\bibitem{CaffarelliCabre1995}
Luis~A. Caffarelli and Xavier Cabr\'{e}, \emph{Fully nonlinear elliptic
  equations}, American Mathematical Society Colloquium Publications, vol.~43,
  American Mathematical Society, Providence, RI, 1995. \MR{1351007}

\bibitem{Ciarlet1978}
Philippe~G. Ciarlet, \emph{The finite element method for elliptic problems},
  Studies in Mathematics and its Applications, vol.~4, North-Holland,
  Amsterdam, 1978.

\bibitem{CrandallIshiiLions1992}
Michael~G. Crandall, Hitoshi Ishii, and Pierre-Louis Lions, \emph{User's guide
  to viscosity solutions of second order partial differential equations}, Bull.
  Amer. Math. Soc. (N.S.) \textbf{27} (1992), no.~1, 1--67. \MR{1118699}

\bibitem{DePhilippisFigalli2013}
Guido De~Philippis and Alessio Figalli, \emph{Second order stability for the
  {M}onge-{A}mp\`ere equation and strong {S}obolev convergence of optimal
  transport maps}, Anal. PDE \textbf{6} (2013), no.~4, 993--1000. \MR{3092736}

\bibitem{DeanGlowinski2006}
E.~J. Dean and R.~Glowinski, \emph{Numerical methods for fully nonlinear
  elliptic equations of the {M}onge-{A}mp\`ere type}, Comput. Methods Appl.
  Mech. Engrg. \textbf{195} (2006), no.~13-16, 1344--1386.

\bibitem{FengJensen2017}
Xiaobing Feng and Max Jensen, \emph{Convergent semi-{L}agrangian methods for
  the {M}onge-{A}mp\`ere equation on unstructured grids}, SIAM J. Numer. Anal.
  \textbf{55} (2017), no.~2, 691--712. \MR{3623696}

\bibitem{Figalli2017}
Alessio Figalli, \emph{The {M}onge-{A}mp\`ere equation and its applications},
  Zurich Lectures in Advanced Mathematics, European Mathematical Society (EMS),
  Z\"{u}rich, 2017. \MR{3617963}

\bibitem{gt2021}
D.~Gallistl and N.~T. Tran, \emph{Convergence of a regularized finite element
  discretization of the two-dimensional {M}onge--{A}mp\`ere equation}, Math.
  Comp. (2022), Accepted for publication, arXiv:2112.10711.

\bibitem{Gutierrez2016}
Cristian~E. Guti\'{e}rrez, \emph{The {M}onge-{A}mp\`ere equation}, Progress in
  Nonlinear Differential Equations and their Applications, vol.~89,
  Birkh\"{a}user/Springer, [Cham], 2016, Second edition [of MR1829162].
  \MR{3560611}

\bibitem{Krylov1987}
N.~V. Krylov, \emph{Nonlinear elliptic and parabolic equations of the second
  order}, Mathematics and its Applications (Soviet Series), vol.~7, D. Reidel
  Publishing Co., Dordrecht, 1987, Translated from the Russian by P. L.
  Buzytsky [P. L. Buzytski\u{\i}]. \MR{901759}

\bibitem{MaugeriPalagachevSoftova2000}
A.~Maugeri, D.~K. Palagachev, and L.~G. Softova, \emph{Elliptic and parabolic
  equations with discontinuous coefficients}, Wiley-VCH Verlag Berlin GmbH,
  Berlin, 2000.

\bibitem{Rockafellar1970}
R.~Tyrrell Rockafellar, \emph{Convex analysis}, Princeton Mathematical Series,
  No. 28, Princeton University Press, Princeton, N.J., 1970. \MR{0274683}

\bibitem{Safonov1988}
M.~V. Safonov, \emph{Classical solution of second-order nonlinear elliptic
  equations}, Izv. Akad. Nauk SSSR Ser. Mat. \textbf{52} (1988), no.~6,
  1272--1287, 1328. \MR{984219}

\bibitem{SmearsSueli2013}
Iain Smears and Endre S\"{u}li, \emph{Discontinuous {G}alerkin finite element
  approximation of nondivergence form elliptic equations with {C}ord\`es
  coefficients}, SIAM J. Numer. Anal. \textbf{51} (2013), no.~4, 2088--2106.
  \MR{3077903}

\bibitem{SmearsSueli2014}
\bysame, \emph{Discontinuous {G}alerkin finite element approximation of
  {H}amilton-{J}acobi-{B}ellman equations with {C}ordes coefficients}, SIAM J.
  Numer. Anal. \textbf{52} (2014), no.~2, 993--1016. \MR{3196952}

\end{thebibliography}
\bibliographystyle{amsplain}

\end{document}